\documentclass[11pt,a4paper,reqno]{article}
\usepackage[latin1]{inputenc}  
\usepackage[T1]{fontenc}
\usepackage[english]{babel}
\usepackage{graphicx}
\usepackage{amsfonts,amssymb,latexsym,a4wide,xspace}
\usepackage{amsmath,delarray,enumerate,calc,amsthm}
\usepackage{bbm}
\usepackage{paralist}
\usepackage{authblk}

\newcommand{\basin}{\B}
\newcommand{\const}{{\rm const}}

\newcommand{\cB}{{\mathcal B}}
\newcommand{\cF}{{\mathcal F}}

\newcommand{\cK}{{\mathcal K}}

\newcommand{\cP}{{\mathcal P}}

\newcommand{\cL}{{\mathcal L}}

\newcommand{\cS}{{\mathcal S}}

\newcommand{\cU}{{\mathcal U}}
\newcommand{\B}{{\mathbbm B}}
\newcommand{\I}{{\mathbbm I}}
\newcommand{\J}{{\mathbbm J}}
\newcommand{\R}{{\mathbbm R}}
\newcommand{\N}{{\mathbbm N}}

\renewcommand{\P}{{\mathbbm P}}
\newcommand{\inn}{\operatorname{int}}

\renewcommand{\S}{{S}}
\newcommand{\mac}{{\mu_{\rm ac}}}

\newcommand{\myDelta}{{H}}
\newcommand{\CH}{C^{1+\text{H\"o{}lder}}}
\newtheorem {definition}{Definition}[section] 
\newtheorem {lemma}[definition]{Lemma}
\newtheorem{theorem}{Theorem}
\newtheorem {bemerkung}[definition]{Remark}
\newtheorem{proposition}[definition]{Proposition}

\newtheorem{beispiel}[definition]{Example}
\newtheorem{hypothese}{Hypothesis}
\newenvironment{remark} {\begin{bemerkung} \normalfont }{\end{bemerkung}}
\newenvironment{example} {\begin{beispiel} \normalfont }{\end{beispiel}}
\newenvironment{hypothesis} {\begin{hypothese} \normalfont }{\end{hypothese}}

\begin{document}

\title{Stability index, uncertainty exponent, and thermodynamic formalism for intermingled basins of chaotic attractors}
\author{Gerhard Keller\thanks{This work was funded by  DFG grant Ke 514/8-1. It also profited from the activities of the DFG Scientific Network ``Skew Product Dynamics and Multifractal Analysis'' organized by Tobias Oertel-J\"ager.}}
\affil{Department Mathematik, Universit\"at Erlangen-N\"urnberg}
\date{\today}


\maketitle

\begin{abstract}
Skew product systems with monotone one-dimensional fibre maps driven by piecewise expanding Markov interval maps may show the phenomenon of intermingled basins \cite{Alexander1992,Bonifant2008,Kan1994,Sommerer1993}. To quantify the degree of intermingledness the uncertainty exponent \cite{Ott1993} and the stability index \cite{Podvigina-Ashwin2011,Roslan-thesis} were suggested and characterized (partially). Here we present an approach to evaluate/estimate these two quantities rigorously using thermodynamic formalism for the driving Markov map.
\end{abstract}

\section{Introduction}
Motivated by numerical observations by Alexander, Yorke, You and Kan \cite{Alexander1992} and Sommerer and Ott \cite{Sommerer1993}, the term \emph{intermingled basins} was introduced to the mathematics literature by Kan in \cite{Kan1994}: ``[\ldots] any open set which intersects one basin in a set of positive measure also intersects each of the other basins in a set of positive measure'', and he showed the existence of an open set of diffeomorphisms with two intermingled basins on a three-dimensional manifold. The degree of intermingledness was soon quantified by the \emph{uncertainty exponent} \cite{Ott1993} originally introduced in \cite{Grebogi1985}. 
In \cite{Ott1994a}  this exponent is calculated numerically for a certain model of point particle motion subject to friction and periodic forcing in a two-dimensional potential (which does not display intermingled but riddled basins),
and the authors support their findings theoretically by calculations carried out for a simple piecewise linear model, 
that has the form of a skew product over two-leg base map with two full linear branches and
for which the dynamics are essentially equivalent to that of a random walk with two regimes of transition probabilities.

Later, a related quantity denoted \emph{stability index} was introduced in \cite{Podvigina-Ashwin2011} and further studied in \cite{Roslan-thesis}. While the uncertainty exponent measures the degree of intermingledness averaged over points that are equidistant to one of the attractors, the stability index measures intermingledness close to individual points, so it resembles a local dimension.

In the present paper we provide rigorous derivations of the two exponents mentioned before for nonlinear generalizations of the piecewise linear system. We use thermodynamic formalism and large deviations theory for the base map in order to study a nonlinear version of the two-regime random walk. The class of systems we investigate includes examples from \cite{Bonifant2008} and also some of the maps from \cite{Hofbauer2004} that serve as models of \emph{indeterminate competition} in two species systems.
For a particular family of such systems the numerical analyses in the spirit of \cite{Ott1994a} was carried out in \cite{Pereira2008}.


In the remainder of this  section we introduce the class of skew product danamical systems studied in this paper and some basic concepts related to them. In Section 2 we recall the characterization of intermingledness through normal Lyapunov exponents, define the stability index and the uncertainty exponent and state our three main theorems, whose proofs are deferred to Section 3.
\subsection{The basic model}
We study skew product dynamical systems $F:\Omega\times\J\to\Omega\times\J$, $F(\omega,x)=(S\omega,f_\omega(x))$, 
defined as follows:
\begin{definition}
$\cF_s$ denotes the family of all skew product transformations $F$ on $\Omega\times\J$ where
\begin{compactenum}[-]
\item $(\Omega,\cB)$ is a standard Borel space and $\S:\Omega\to\Omega$ is a $\cB$-measurable transformation,
\item $\J$ is a compact interval that we choose as $\J=[-1,1]$ for convenience,
\item the maps $f_\omega:\J\to\J$ are diffeomorphisms
with $f_\omega(\pm1)=\pm1$ and negative Schwarzian derivative $\cS f_\omega<0$, where 
$\cS f:=\frac{f'''}{f'}-\frac{3}{2}\left(\frac{f''}{f'}\right)^2$,
\item
and finally we require that 
$(\omega,x)\mapsto f_\omega(x)$ is measurable.
\end{compactenum}
\end{definition}

For the iterates $F^n$ of $F$ we adopt the usual notation $F^n(\omega,x)=(\S^n\omega,f_\omega^n(x))$. 
Hence $f_\omega^{n+k}(x)=f_{\S^k\omega}^n(f_\omega^k(x))$.
For $n=1$ and $k=-1$ this includes the identity
$f_\omega^{-1}(x)=(f_{\S^{-1}\omega})^{-1}(x)$.

\subsection{Invariant graphs and their Lyapunov exponents}
Denote by $\cP(S)$ the space of all $S$-invariant probability measures on $(\Omega,\cB)$ and  by $\cP_e(\S)$ the family of all ergodic $\nu\in\cP(S)$.
As $(\Omega,\cB)$ is assumed to be standard Borel, non-ergodic invariant probability measures can be decomposed into ergodic ones.

\begin{definition}
Let $F\in\cF_s$. A measurable function $\varphi:\Omega\to \J $ is an \emph{invariant graph}, if
\begin{equation*}
F(\omega,\varphi(\omega))=(\S\omega,\varphi(\S\omega)),\;\text{equivalently }
f_\omega(\varphi(\omega))=\varphi(\S\omega),\quad\text{for all }\omega\in\Omega\ .
\end{equation*}
If $\nu\in\cP(S)$ and the identity holds for $\nu$-almost every $\omega$ we call $\varphi$ a \emph{$\nu$-a.e. invariant graph}.
(The slight abuse of terminology underlines that we are mostly interested in properties of the graph of $\varphi$ as a subset of $\Omega\times\J$.)
\end{definition}
Observe that each $F\in\cF_s$ has the two constant invariant graphs $\varphi^\pm(\omega):=\pm1$.

\begin{definition}
\begin{enumerate}[a)]
\item
Let $F\in\cF_s$ and $(\omega,x)\in\Omega\times \J $. If the limit
\begin{equation*}
\lambda(\omega,x):=\lim_{n\to\infty}\frac{1}{n}\log|Df^n_\omega(x)|
\end{equation*}
exists,it is called the normal Lyapunov exponent of $F$ in $(\omega,x)$. 
\item 
Let $\nu\in\cP_e(\S)$.
If $\varphi$ is a $\nu$-a.e. invariant graph with $\log|Df_\omega(\varphi(\omega))|\in \cL^1_\nu$, its Lyapunov exponent w.r.t. $\nu$ is defined as
\begin{equation*}
\lambda_\nu(\varphi):=\int_\Omega\log|Df_\omega(\varphi(\omega))|\,d\nu(\omega)\ .
\end{equation*}
\end{enumerate} 
\end{definition}

The following proposition is a consequence of negative Schwarzian derivative.
\begin{proposition}[\cite{Jager2003,Bonifant2008}]
\label{prop:JBM}
Let $F\in\cF_s$ and $\nu\in\cP_e(S)$. Then 
$\lambda_\nu(\varphi^-)+\lambda_\nu(\varphi^+)<0$.
\end{proposition}

\begin{definition}
The basin of $\varphi^\pm$ is 
\begin{equation*}
\basin^\pm:=\left\{(\omega,x)\in\Omega\times\J:\,|f_\omega^n(x)-\varphi^\pm|\to0\ (n\to\infty)\right\}.
\end{equation*}
Denote $\varphi_c^-(\omega):=\sup\{x\in\J:\,(\omega,x)\in\basin^-\}$
and $\varphi_c^+(\omega):=\inf\{x\in\J:\,(\omega,x)\in\basin^+\}$, and observe that $\varphi_c^-$ and $\varphi_c^+$ are invariant graphs.
\end{definition}

Part a) of the following proposition can be found in \cite{Bonifant2008}. In a slightly different setting, both parts were proved in the earlier
paper \cite{Jager2003}. The reader can easily adapt the proof of part b) to the present setting.

\begin{proposition}\label{prop:two-basins}
Let $F\in\cF_s$, $\nu\in\cP_e(\S)$. Then exactly one of the following three possibilities occurs:
\begin{enumerate}[a)]
\item
If
$\lambda_\nu(\varphi^-)<0$ and $\lambda_\nu(\varphi^+)<0$,
then $\varphi^-<\varphi_c^-=\varphi_c^+<\varphi^+$ $\nu$-a.e. so that $(\nu\times m)((\Omega\times J)\setminus(\basin^-\cup\basin^+))=0$, and $(\nu\times m)(\basin^\pm)>0$. 
\item
If
$\lambda_\nu(\varphi^-)<0\leqslant\lambda_\nu(\varphi^+)$,
then $\varphi^-<\varphi_c^-=\varphi_c^+=\varphi^+$ $\nu$-a.e. so that $(\nu\times m)((\Omega\times J)\setminus\basin^-)=0$, and $(\nu\times m)(\basin^+)=0$. 
\item
If
$\lambda_\nu(\varphi^+)<0\leqslant\lambda_\nu(\varphi^-)$,
then $\varphi^-=\varphi_c^-=\varphi_c^+<\varphi^+$ $\nu$-a.e. so that $(\nu\times m)((\Omega\times J)\setminus\basin^+)=0$, and $(\nu\times m)(\basin^-)=0$. 
\end{enumerate}
In particular, there are at most three $\nu$-equivalence classes of invariant graphs, namely $\varphi^-$, $\varphi^+$ and $\varphi_c$.
\end{proposition}
The fact that there are at most three $\nu$-equivalence classes of invariant graphs is, of course, due to the assumption of negative Schwarzian derivative.

\section{Stability index and uncertainty exponent for intermingled basins}
In order to keep the notation as simple as possible, we specialize already from here on to one-dimensional mixing Markov maps at the base:
\begin{hypothesis}\label{hypo:1}
$\Omega=\I$ is a finite interval, and $S:\I\to\I$
is a piecewise expanding and piecewise $\CH$ mixing Markov map 
with finitely many branches. It is a well known fact, often called the folklore theorem \cite{Bowen79}, that $S$ has a unique invariant probability measure $\mac$ equivalent to Lebesgue measure $m$ on $\I$.\footnote{Indeed, early sources treat only the piecewise $C^2$ case, but the proof carries over to piecewise $\CH$ maps.}
\end{hypothesis}

\subsection{Intermingled basins}

An early reference for intermingled basins (with quadratic fibre maps) is the paper \cite{Kan1994}.

\begin{definition}\label{def:intermingled}
Let $F\in\cF_s$. We say that the two basins 
$\basin^+$ and $\basin^-$ are \emph{intermingled}, if the intersections of both basins with any open set has positive Lebesgue-measure.
\end{definition}

\begin{proposition}[\cite{Bonifant2008}]\label{prop:intermingled}
If $\lambda_\mac(\varphi^-)<0$ and $\lambda_\mac(\varphi^+)<0$ and
if there are $\nu^-,\nu^+\in\cP_e(S)$ such that $\lambda_{\nu^-}(\varphi^-)>0$ and $\lambda_{\nu^+}(\varphi^+)>0$, then $\basin^-$ and $\basin^+$ are intermingled.
\end{proposition}
\noindent This is precisely the situation we are looking at, so we state this as a further hypothesis.
\begin{hypothesis}\label{hypo:2}
$\lambda_\mac(\varphi^-)<0$ and $\lambda_\mac(\varphi^+)<0$ , and
there are $\nu^-,\nu^+\in\cP_e(S)$ such that $\lambda_{\nu^-}(\varphi^-)>0$ and $\lambda_{\nu^+}(\varphi^+)>0$ (so that
Proposition~\ref{prop:intermingled} applies).
\end{hypothesis}

\begin{example}\label{ex:1}
Let $A,B:\Omega\to\R$, $A(\omega)=a+b\cos(2\pi\omega)$, $B(\omega)=c\sin(2\pi\omega+d)$ with $|b|<|a|$, and define
\begin{displaymath}
f_\omega(x):=\frac{\arctan(A(\omega)\cdot x+B(\omega))-D(\omega)}{C(\omega)}
\end{displaymath}
where $C(\omega)$ and $D(\omega)$ are determined such that
$f_\omega(\pm1)=\pm1$. Then $\cS f_\omega<0$.

For a numerical example  we use the doubling map as base transformation $\S$ that leaves the Lebesgue measure invariant, and we choose the following parameters:
$a\in[0.8,5],\; b=0.4,\;c=0.9,\;d=0.8$.
This choice guarantees
that $\lambda_m(\varphi^-)<0$ and
$\lambda_m(\varphi^+)=<0$.
When $a\in[0.8,1]$, then
the Lyapunov-exponent of the 4-periodic point $\omega_4=0.0666$ on the circle $x=-1$ and
that of the 5-periodic point $\omega_5=0.3548$ on the circle $x=1$
are strictly positive.
The equidistributions on these orbits thus play the role of the measures $\nu^-$ and $\nu^+$.

It seems that for $a\geqslant1.2$ the lower and upper invariant graph are uniformly attracting.
Hence there is $\delta>0$ such that $[0,1]\times[-1+\delta,1-\delta]$ is invariant under $F^{-1}$, and all $F^{-1}$-invariant measures are supported by $\varphi_c$ and have a negative vertical Lyapunov exponent under $F^{-1}$. Therefore Lemma 5.1 in \cite{Anagnostopoulou} suggests that $\varphi_c$ is continuous. (We say ``suggests'', because the baker map at the base is not a homeomorphism as required in that Lemma.)
\begin{figure}\label{fig:1}
\centering{
\includegraphics[height=10cm,width=15cm]{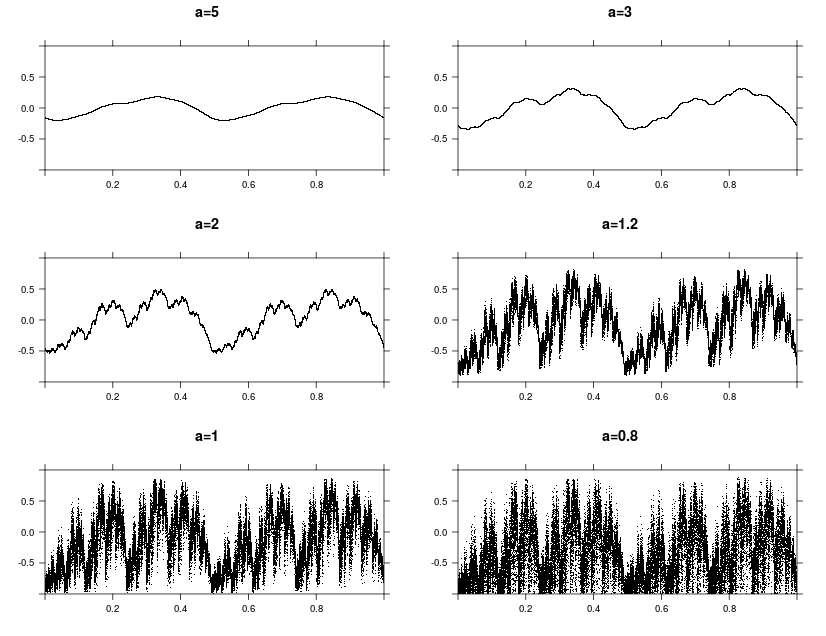}
}
\caption{The critical graphs $\varphi_c$ for various choices of the parameter $a$.}
\end{figure}
In those cases where $\varphi_c$
is a continuous curve bounded away from $\varphi^-$ and from $\varphi^+$, the box dimension of the graph of $\varphi_c$ is known: Bedford \cite{Bedford1999} studied systems that include our examples and showed that the box dimension is the unique zero of the pressure function 
\begin{equation*}
t\mapsto p\big(\omega\mapsto -(t-1)\log |S'(\omega)|+\log f_\omega'(\varphi_c(\omega)))\ ,
\end{equation*}
provided this zero is bigger than $1$.
\end{example}

\subsection{The stability index}

As in \cite{Keller2012a}, which was strongly motivated by  \cite{Podvigina-Ashwin2011}, we define a local stability index $\sigma(\omega,x)$ of a point $(\omega,x)\in\I\times\J$ w.r.t. $\basin^-$ and $\basin^+$ in the following way: Let
\begin{equation}
\Sigma^\pm_\epsilon(\omega,x):=
\frac{m^{2}\left(U_\epsilon(\omega,x)\cap\basin^\pm\right)}{m^{2}\left(U_\epsilon(\omega,x)\right)}\  ,
\end{equation}
where $U_\epsilon(\omega,x)=[\omega-\epsilon,\omega+\epsilon]\times[x-\epsilon,x+\epsilon]$.
Obviously, $\Sigma^{-}_\epsilon(\omega,x)+\Sigma^{+}_\epsilon(\omega,x)=1$.  Define
\begin{equation}
\sigma(\omega,x):=\sigma^+(\omega,x)-\sigma^-(\omega,x)
\end{equation}
where
\begin{equation}\label{eq:local-scaling}
\sigma^\pm(\omega,x):=\lim_{\epsilon\to0}\frac{\log\Sigma^\pm_\epsilon(\omega,x)}{\log\epsilon}
\end{equation}
Of course, the limits in (\ref{eq:local-scaling}) need not exist for every $(\omega,x)$, but
observe that $\liminf_{\epsilon\to0}\frac{\log\Sigma^\pm_\epsilon(\omega,x)}{\log\epsilon}$ $\geqslant0$ always.
Hence $\sigma^+(\omega,x)$ and $\sigma^-(\omega,x)$ are non-negative, and at most one of them can be strictly positive.
Observe also that
$\sigma^\pm(F(\omega,x))=\sigma^\pm(\omega,x)$ for all $(\omega,x)\in\I\times\J$.
(This is essentially Theorem 2.2 of \cite{Podvigina-Ashwin2011}.)
\begin{remark}
Denote by $m^\pm$ the {$2$-dimensional} Lebesgue measure on $\Omega\times\J$ restricted to $\basin^\pm$ Then $\sigma^\pm(\omega,x)+2$ is just the local dimension of the measure $m^\pm$ at $(\omega,x)$.
\end{remark}

In order to obtain precise quantitative information about $\sigma^\pm$, we require some regularity of $f_\omega(x)$:
\begin{hypothesis}\label{hypo:3}
The function $(\omega,x)\mapsto \log f'_\omega(x)$ is $\alpha$-H\"older continuous on each set $I_i\times\J$ where $I_i$ is a Markov interval of $\S$.
\end{hypothesis}

A first step towards evaluating $\sigma^\pm$ is the following result which is similar to
Theorem 2.1 in \cite{Keller2012a}. Its proof is provided in Section~\ref{subsec:proofTheo-exponential}.

\begin{theorem}
\label{theo:exponential}
Let $F\in\cF_s$ and assume Hypotheses~\ref{hypo:1} -~\ref{hypo:3}. There are $t^-_*,t^+_*>0$ (both defined by thermodynamic formalism, see below) such that
\begin{equation*}
\lim_{\epsilon\to0}\frac{\log m\{\varphi_c<\varphi^-+\epsilon)\}}{\log \epsilon}=t^-_*
\quad\text{and}\quad
\lim_{\epsilon\to0}\frac{\log m\{\varphi_c>\varphi^+-\epsilon\}}{\log \epsilon}=t^+_*\ .
\end{equation*}
\end{theorem}
The numbers $t^\pm_*$ are uniquely determined as positive zeros of the pressure functions 
\begin{equation}\label{eq:one-parameter-pressure}
\begin{split}
t\mapsto p^\pm(t):= &p(-\log |\S'|+t\log f'_\bullet(\varphi^\pm))\\
=& \sup_{\nu\in\cP(\S)}\left(h_\nu(\S)-\int\log|S'|\,d\nu+t\lambda_\nu(\varphi^\pm)\right). 
\end{split}
\end{equation}
({To be more precise, we mean the pressure function of the topological Markov chain that encodes $S$.})
Indeed, $p^\pm(0)=0$ and $(p^\pm)'(0)=\lambda_\mac(\varphi^\pm)<0$ by Hypothesis~\ref{hypo:2} (see e.g. \cite{ParryPollicott,keller-book}), so that the convex functions $p^-$ and $p^+$ have  unique positive zeros provided $\sup_{s>0} p^+(s)>0$ and $\sup_{s>0}p^-(s)>0$. This latter property is guaranteed by the existence of the measures $\nu^-$ and $\nu^+$ from Hypothesis~\ref{hypo:2}. 

\begin{remark}\label{remark:comparison1}
This formula for $t_*^\pm$ is in good agreement with the formula for the corresponding quantity $\eta$ in \cite{Ott1994a,Ott1993}. As the authors work with a diffusion approximation and restrict to negative $\lambda_\mac(\varphi^\pm)$ close to zero, $t_*^\pm$ can be approximated by the positive zero of the log-Laplace transform of a normal distribution with mean $\lambda_\mac(\varphi^\pm)=(p^\pm)'(0)$ and variance $\sigma^2:=(p^\pm)''(0)$. This log-Laplace transform has the form $t\mapsto \lambda_\mac(\varphi^\pm)\cdot t+\frac{\sigma^2t^2}{2}$ -- so its positive zero is $-\frac{2\lambda_\mac(\varphi^\pm)}{\sigma^2}=\frac{|\lambda_\mac(\varphi^\pm)|}{D}=\eta$, where we used the convention $\sigma^2=2D$ as in \cite{Ott1994a,Ott1993}.
Numerical studies in \cite{Ott1994a,Ott1993} and later in \cite{Pereira2008} confirm this approximate formula for the exponent.
\end{remark}

\begin{theorem}\label{theo:main}
Let $F\in\cF_s$ and assume Hypotheses~\ref{hypo:1} -~\ref{hypo:3}. 
For any Gibbs measure $\nu\in\cP_e(\S)$ denote
$\Delta_\nu(\varphi^\pm):=\int\log|\S'|\,d\nu- \lambda_\nu(\varphi^\pm)$.
Then, if $\nu=\mac$, we have for $\nu$-a.e. $\omega\in\I$,
\begin{equation}\label{eq:cases}
\sigma(\omega,x)=
\begin{cases}
\sigma^+(\omega,x)=t_*^-\cdot\frac{\Delta_\nu(\varphi^-)}{\int\log|\S'|\,d\nu}>0&\text{ if }x=\varphi^-\text{ and }\Delta_\nu(\varphi^-)>0\\
\sigma^+(\omega,x)=t_*^-\cdot\frac{-\lambda_\nu(\varphi^-)}{\int\log|\S'|\,d\nu}>0&\text{ if }x\in(\varphi^-,\varphi_c(\omega))\\
0&\text{ if }x=\varphi_c(\omega)\in(\varphi^-,\varphi^+)\\
-\sigma^-(\omega,x)=t_*^+\cdot\frac{\lambda_\nu(\varphi^+)}{\int\log|\S'|\,d\nu}<0&\text{ if }x\in(\varphi_{c}(\omega),\varphi^+)\\
-\sigma^-(\omega,x)=t_*^+\cdot\frac{-\Delta_\nu(\varphi^+)}{\int\log|\S'|\,d\nu}<0
&\text{ if }x=\varphi^+\text{ and }\Delta_\nu(\varphi^+)>0\\
\end{cases}
\end{equation}
The same formula holds for general Gibbs measures $\nu\in\cP_e(\S)$, except possibly in the case $x=\varphi_c(\omega)\in(\varphi^-,\varphi^+)$.\footnote{If $\nu\in\cP_e(\S)$ is a Gibbs measure with $\lambda_\nu(\varphi_c)\geqslant\int\log|S'|\,d\nu$, then eq.~\eqref{eq:cases} continues to hold also for $x=\varphi_c(\omega)\in(\varphi^-,\varphi^+)$. This follows from eq.~\eqref{eq:sigma+upper} and the one-sided Koebe principle eq.~\eqref{eq:os-Koebe} below, because, with the notation from eq.~\eqref{eq:sigma+upper}, $f_\omega^{n_k}(x+\epsilon_k)-\varphi^-\geqslant\min\left\{1,\frac{\epsilon_k}{2}(f_\omega^{n_k})'(\varphi_c(\omega))\right\}$.}
\end{theorem}

\begin{remark}\label{remark:main}
\begin{enumerate}[a)]
\item The requirement that $\nu\in\cP_e(\S)$ is a Gibbs measure can be considerably relaxed, see \cite[Remark 5]{Keller2012a}.
\item Because of Hypothesis~\ref{hypo:2}, $\Delta_{\mac}(\varphi^\pm)>0$; so for $\nu=\mac$ the above five possibilities form an exhaustive list.
\item Note that $\Delta_\nu(\varphi^-)>0$ in case 2 and $\Delta_\nu(\varphi^+)>0$ in case 4. In case 3, $\lambda_\nu(\varphi^\pm)>0$ by Proposition~\ref{prop:two-basins}, so that $\Delta_\nu(\varphi^\pm)>0$.
\item For general Gibbs measures $\nu\in\cP_e(S)$, I cannot exclude the possibility that $\Delta_\nu(\varphi^-)\leqslant0$ or $\Delta_\nu(\varphi^+)\leqslant0$. In that case the analysis is much more involved. Indeed, 
 if $x=\varphi^-$ and $\Delta_\nu(\varphi^-)\leqslant0$,
one can show that
$\sigma(\omega,x)=-\sigma^-(\omega,x)=\frac{\Delta_\nu(\varphi^-)}{\int\log|\S'|\,d\nu}\left(\frac{t_*^+|\lambda^+|}{\lambda^-}-(t_*^+|\lambda^+|-1)^+\right)$, but I do not attempt to provide a proof in this paper. A similar result holds when
$x=\varphi^+$ and $\Delta_\nu(\varphi^+)\leqslant0$. 
\end{enumerate}
\end{remark}

\subsection{The uncertainty exponent}

Let $F\in\cF_s$ be as before. As in \cite{Grebogi1985,Ott1994a,Ott1993}
we define the uncertainty exponent of $F$ in the following way: Choose $x\in\inn(\J)$ and denote by $\Pi_{\epsilon,x}$ the probability that two  points chosen at random with distance at most $\epsilon$ on the line $\I\times\{x\}$ belong to different basins. More formally,
\begin{equation*}
\begin{split}
\Pi_{\epsilon,x}
&=
(2\epsilon)^{-1}\cdot m^2\left\{(\omega,\omega')\in\I^2: (\omega,x)\in\basin^-, (\omega',x)\in\basin^+, |\omega'-\omega|<\epsilon\right\}\\
&=
(2\epsilon)^{-1}\cdot m^2\left\{(\omega,\omega')\in\I^2: \varphi_c(\omega)>x, \varphi_c(\omega')<x, |\omega'-\omega|<\epsilon\right\}\ .
\end{split}
\end{equation*}

\begin{theorem}\label{theo:uncertainty}
Let $F\in\cF_s$ and assume Hypotheses~\ref{hypo:1} -~\ref{hypo:3}.
For every $x\in\inn(\J)$ we have
\begin{equation*}
\limsup_{\epsilon\to0}\frac{\log\Pi_{\epsilon,x}}{\log\epsilon}
\leqslant
\phi\ ,
\end{equation*}
where $\phi$ is determined by the pressure function
\begin{equation}\label{eq:pressure-function}
(\lambda^-,\lambda^+,\lambda^S)\mapsto p\left(\lambda^-\log f_\bullet'(\varphi^-)+\lambda^+\log f_\bullet'(\varphi^+)+(\lambda^S-1)\log|S'|\right).
\end{equation}
Details on $\phi$ are provided in Equations
(\ref{eq:phi1}) and (\ref{eq:phi}) in Section~\ref{subsec:Proof_Theo-uncertainty}.
\end{theorem}
\begin{remark}
One cannot expect this estimate to be sharp, because the asymptotics of $\Pi_{\epsilon,x}$ depend crucially on the large deviations behaviour of the time a trajectory needs to enter the final regime, when it is attracted to either $\varphi^-$ or $\varphi^+$. But this is governed (at best) by $\lambda_\mac(\varphi_c)$, an exponent of which we do not know more than the bounds $0<\lambda_\mac(\varphi_c)<\min\{-\lambda_\mac(\varphi^-),
-\lambda_\mac(\varphi^+)\}$ (because of negative Schwarzian derivative). As this is the exponent of a measure supported by the graph of the highly discontinuous function $\varphi_c$, it is unlikely to be characterizable by thermodynamic quantities related to the base transformation.

This is only one reason why it would be difficult to compare our $\phi$ to the approximate expression for $\phi$ derived in \cite{Ott1994a,Ott1993} or to the numerical results provided in for $\phi$ provided in \cite{Ott1994a,Ott1993,Pereira2008}.
Further issues are that the authors of \cite{Ott1994a,Ott1993} treat riddled basins, which can be modeled by a single random walk, while we need for the intermingled case two random walks (one for each attractor) that are moreover negatively correlated. 

We return to this in Remark~\ref{remark:Ott} after the proof of Theorem~\ref{theo:uncertainty}, where we derive an approximate formula for $\phi$ that yields a value twice as large as that used in \cite{Ott1994a,Ott1993,Pereira2008}. This might explain why not only in \cite[Figure 13]{Ott1994a}, but also in \cite[Figure 6b (small $\kappa$)]{Pereira2008}, the numerically observed values for $\phi$ are roughly twice as large as the values suggested by the formula used there.
%
\end{remark}

\section{Proofs}

\subsection{Distortion estimates}

\begin{remark}\label{remark:distortion}
Denote
by $\cU_n(\omega)$ the family of all interval neighbourhoods $U$ of $\omega \in\I$ such that $\S^n_{|U}:U\to \S^nU$ is a diffeomorphism.
It is a well known fact (e.g. \cite{CFS-book}) that under Hypotheses~\ref{hypo:1} and~\ref{hypo:2} there is a distortion constant $D\geqslant1$ such that
for all $n>0$, all $\omega\in\I$, all $U\in\cU_n(\omega)$ and all ${\tilde\omega}\in U$
\begin{equation}\label{eq:distortion1}
e^{-D}\leqslant
\left|\frac{(\S^n)'(\tilde\omega)}{(\S^n)'(\omega)}\right|,\;
\left|\frac{(f^n_{\tilde\omega})'(\varphi^\pm)}{(f^n_{\omega})'(\varphi^\pm)}\right|
\leqslant e^D
\end{equation}
and \cite[Lemma 2.6]{Ott2009}
\begin{equation}\label{eq:distortion2}
\left|\log\frac{(f^n_{\tilde\omega})'(\tilde{x})}{(f^n_{\omega})'({x})}\right|
\leqslant D\cdot\left(|\S^n(\tilde{\omega})-\S^n(\omega)|+
\sum_{i=0}^{n-1}|f_{\omega}^i(\tilde{x})-f_\omega^i(x)|^\alpha\right)
\quad\text{for all  $x,\tilde{x}\in\J$.} 
\end{equation}
\end{remark}
The following lemma is an immediate consequence of (\ref{eq:distortion2}).
\begin{lemma}\label{lemma:dist-folklore}
If $\lim_{n\to\infty}|f_\omega^n(\tilde x)-f_\omega^n(x)|=0$, then, for each $\delta>0$ there is $C=C(x,\tilde{x},\omega,\delta)>0$ such that
\begin{displaymath}
|\tilde{x}-x|\,(f_\omega^n)'(x)\,e^{-C-n\delta}
\leqslant
|f_\omega^n(\tilde{x})-f_\omega^n(x)|
\leqslant
|\tilde{x}-x|\,(f_\omega^n)'(x)\,e^{C+n\delta}
\end{displaymath}
for all $n\in\N$.
\end{lemma}
Another consequence is:
\begin{lemma}\label{lemma:dist-prep}
Let $n>0$, $\omega\in\I$ and $U\in\cU_n(\omega)$. Then
\begin{equation}\label{eq:distortion3}
e^{-D\cdot|\S^n(\tilde{\omega})-\S^n(\omega)|}
\leqslant
\frac{f_{\tilde{\omega}}^n(\tilde x)-f_{\tilde{\omega}}^n(x)}{f_\omega^n(\tilde x)-f_\omega^n(x)}\leqslant e^{D\cdot|\S^n(\tilde{\omega})-\S^n(\omega)|}
\quad\text{for all  $x,\tilde{x}\in\J$ and ${\tilde\omega}\in U$,} 
\end{equation}
and if $h:\S^n(U)\to\J$ is such that $\{(z,h(z)):z\in\S^n(U)\}=F^n(U\times\{x\})$,
i.e. $h(\S^n(\tilde{\omega}))=f_{\tilde{\omega}}^n(x)$,
 then $\left|\log\frac{h(z)-\varphi^-}{h(\S^n(\omega))-\varphi^-}\right|\leqslant D\cdot|z-\S^n(\omega)|$ 
for all $z\in \S^n(U)$ 
 (and similarly for $\varphi^+$).
\end{lemma}
\begin{proof} (\ref{eq:distortion3}) follows from (\ref{eq:distortion2}), because
\begin{equation*}
f_{\tilde{\omega}}^n(\tilde x)-f_{\tilde{\omega}}^n(x)
=
\int_x^{\tilde{x}}(f_{\tilde{\omega}}^n)'(u)\,du
=
\int_x^{\tilde{x}}(f_{\omega}^n)'(u)\,\frac{(f_{\tilde{\omega}}^n)'(u)}{(f_{\omega}^n)'(u)}\,du\ .
\end{equation*}
Setting $\tilde{x}=\varphi^-$ in (\ref{eq:distortion3}) shows that 
$
\left|\log\frac{h(\S^n(\tilde{\omega}))-\varphi^-}{h(\S^n(\omega))-\varphi^-}\right|
\leqslant D\cdot|\S^n(\tilde{\omega})-\S^n(\omega)|
$. 
\end{proof}

\begin{lemma}\label{lemma:LD-prep1}
Let $\delta>0$.
There are $n_0\in\N$ and $\delta'\in(0,\delta)$ which depend only on $\delta$ with the following property: For all $\omega\in\I$, $x\in\J$, $\vartheta\in(-\delta',\delta')$ and $n\geqslant n_0$ such that 
$|\vartheta|\,(f^k_\omega)'(x)\leqslant e^{-2k\delta}$ $(k=n_0,\dots,n)$ holds:
\begin{equation}\label{eq:dist-est}
|\vartheta|\cdot (f^n_\omega)'(x)
\cdot e^{-3\delta n}
\leqslant
|f_\omega^n(x+\vartheta)-f_\omega^n(x)|
\leqslant
|\vartheta|\cdot (f^n_\omega)'(x)
\cdot e^{\delta n}
\end{equation}
\end{lemma}

\begin{proof}
Choose $n_0$ such that $e^{-n_0\delta}<\left(\frac\delta D\right)^{1/\alpha}$, where $\alpha$ 
is the H\"older exponent from Hypothesis~\ref{hypo:3} and $D\geqslant1$ is the distortion constant  from Remark~\ref{remark:distortion}.
Hypothesis~\ref{hypo:3} allows in particular to  fix $\delta'\in(0,\delta)$ such that $|f^n_\omega(x+\vartheta)-f^n_\omega(x)|\leqslant\left(\frac{\delta}D\right)^{1/\alpha}$ for all $n=0,\dots,n_0-1$, all $\omega\in\I$, $x\in\J$ and all $\vartheta\in(-\delta',\delta')$.

Now consider $\omega$, $x$ and $\vartheta$ as in the assumption of the lemma and denote $\myDelta_k:=f_\omega^k(x+\vartheta)-f_\omega^k(x)$. As $\vartheta\in(-\delta',\delta')$, we have $|\myDelta_k|^\alpha\leqslant\frac{\delta}{D}$ for all $k=1,\dots,n_0-1$. For $k=n_0,\dots,n$ there is $z_k$ between $x$ and $x+\vartheta$ such that
\begin{equation*}
|\myDelta_k|
=
|\vartheta|\cdot(f^k_\omega)'(z_k)
=
|\vartheta|\cdot(f^k_\omega)'(x)\cdot\exp\left(\log\frac{(f^k_\omega)'(z_k)}{(f^k_\omega)'(x)}\right).
\end{equation*}
Denote $\kappa_k:=|\vartheta|\,(f^k_\omega)'(x)\, e^{2k\delta}$ and observe that $\kappa_k\leqslant1$ by the assumptions of the lemma. Then we have in view of (\ref{eq:distortion2})
\begin{equation}\label{eq:dist-intermediate}
|\myDelta_k|
\leqslant
\kappa_k
\cdot e^{-2k\delta}\cdot\exp\left(D\cdot\sum_{i=0}^{k-1}|f_\omega^i(z_k)-f_\omega^i(x)|^\alpha
\right)
\leqslant
\kappa_k\cdot e^{-2k\delta}\cdot\exp\left(D\cdot\sum_{i=0}^{k-1}|\myDelta_i|^\alpha
\right).
\end{equation}
Recall that $D|\myDelta_i|^\alpha\leqslant\delta$ for $i=0,\dots,n_0-1$, so that $D\sum_{i=0}^{k-1}|\myDelta_i|^\alpha
\leqslant
n_0\delta+\sum_{i=n_0}^{k-1}D|\myDelta_i|^\alpha$.
Hence, for $k=n_0,\dots,n$,
\begin{equation*}
\begin{split}
|\myDelta_k|
&\leqslant
\kappa_k\cdot \exp\left(-k\delta+\sum_{i=n_0}^{k-1}(D|\myDelta_i|^\alpha-\delta)\right).
\end{split}
\end{equation*}
For $k=n_0,\dots,n$ it follows inductively that
$D|\myDelta_k|^\alpha\leqslant
D\kappa_k^\alpha e^{-k\delta\alpha}\leqslant D e^{-n_0\delta\alpha}<
\delta$ (recall that $\kappa_k\leqslant1$). Hence 
\begin{equation*}
|\myDelta_n|\leqslant\kappa_n e^{-n\delta}
=
|\vartheta|\,(f^n_\omega)'(x)\, e^{n\delta},
\end{equation*}
which is the upper estimate in (\ref{eq:dist-est}). For the lower estimate observe that the reasoning leading to (\ref{eq:dist-intermediate}) also yields
\begin{equation*}
|\myDelta_n|
\geqslant
\kappa_n\cdot e^{-2n\delta}\cdot\exp\left(-D\cdot\sum_{i=0}^{n-1}|\myDelta_i|^\alpha
\right)
\end{equation*}
and that $D|\myDelta_i|^\alpha<\delta$ for $i=0,\dots,n-1$, as we just showed.
\end{proof}

\begin{remark}\label{remark:LD-prep1}
As the proof of this lemma did not make use of negative Schwarzian derivative, it applies as well to the inverse branches $(f_\omega^n)^{-1}$.
\end{remark}

\subsection{Proof of Theorem~\ref{theo:exponential}}\label{subsec:proofTheo-exponential}

We follow closely \cite[Sections 4.1 and 7]{Keller2012a}:
For $t\in\R$ denote by $\cL_t^\pm$ the transfer operators
\begin{equation}\label{eq:PF}
\cL_t^\pm:L^1_m(\I)\to L^1_m(\I),\; 
\cL_t^\pm f(\omega)=\sum_{\tilde{\omega}\in S^{-1}\omega}\frac{f(\tilde{\omega})}{|S'(\tilde{\omega})|}e^{t \log f_{\tilde{\omega}}'(\varphi^\pm)}\ ,
\end{equation}
and let $\rho(\cL_t^\pm)$ be its spectral radius. Then $p^\pm(t)=\log\rho(\cL_t^\pm)$, and this is a strictly convex differentiable function of $t$, see e.g.
\cite{ParryPollicott}. 

We prove only the identity for the limit that evaluates to $t_*^-$, the other one is proved in the same way. So fix any $t\in(0,{t_*^-})$ and choose $\delta>0$ such that $\rho(\cL_t^-)e^{4t\delta}<1$.
There is a constant $C=C_{t,\delta}>0$ such that 
\begin{equation}\label{eq:spectral-radius}
\|(\cL_t^-)^n1\|_1\leqslant C\left(\rho(\cL^-_t)e^{t\delta}\right)^n\leqslant C e^{-3nt\delta}\text{\quad for all }n\geqslant 1\ .
\end{equation}

\begin{lemma}\label{lemma:LD-prep2}
Let $t\in(0,t_*^-)$ and $\delta>0$ be as chosen above. Then
\begin{equation*}
m\left\{\omega\in\I: \vartheta\cdot (f^n_\omega)'(\varphi^-)\geqslant e^{-2n\delta}\right\}
\leqslant
C e^{-nt\delta}\vartheta^{t}
\end{equation*}
for all $\vartheta>0$ and $n\geqslant1$.
\end{lemma}
\begin{proof}
As $t>0$, we have the usual Cram\'er type estimate for each $n\geqslant1$:
\begin{align*}
m\left\{\omega\in\I: \vartheta\cdot(f^n_\omega)'(\varphi^-)\geqslant e^{-2n\delta}\right\}
&=
m\left\{\omega\in \I:
\vartheta^t e^{2nt\delta}
\,
e^{t\log (f^n_\omega)'(\varphi^-)}\geqslant1
\right\}
\\
&\leqslant
\vartheta^t e^{2nt\delta}
\int_\I e^{t\log (f^n_\omega)'(\varphi^-)}\,dm\\
&=
\vartheta^t e^{2nt\delta}
\int_\I(\cL_0^-)^n(e^{t\log (f^n_\omega)'(\varphi^-)})\,dm\\
&=
\vartheta^t e^{2nt\delta}
\int_\I(\cL_t^-)^n(1)\,dm\\
&\leqslant
C e^{-nt\delta}\cdot\vartheta^t\ ,
\end{align*}
where we used (\ref{eq:spectral-radius}) for the last inequality.
\end{proof}

\begin{proof}[Proof of Theorem~\ref{theo:exponential},
lower bound: 
$\limsup_{\epsilon\to0}\frac{\log m\{\varphi_c<\varphi^-+\epsilon\}}{\log \epsilon}\geqslant t$ for each $t\in(0,t_*^-)$]\quad\\
Let $n_0$ and $\delta'$ be as in Lemma~\ref{lemma:LD-prep1} and consider $\epsilon\in(0,\delta')$. Then
\begin{equation*}
\begin{split}
m\left\{\omega\in\I:\varphi_c(\omega)<\varphi^-+\epsilon \right\}
&\leqslant
m\left\{\omega\in\I: \exists n\geqslant n_0 \text{ s.t }
f_\omega^n(\varphi^-+\epsilon)
>
\varphi^-+e^{-\delta n}\right\}\\
&\leqslant
m\left\{\omega\in\I: \exists k\geqslant n_0 \text{ s.t }
\epsilon\cdot (f^k_\omega)'(\varphi^-)> e^{-2k\delta}\right\}
\end{split}
\end{equation*}
in view of the upper estimate in Lemma~\ref{lemma:LD-prep1}. Hence 
$m\left\{\omega\in\I:\varphi_c(\omega)<\varphi^-+\epsilon \right\}
\leqslant C'_{t,\delta}\,\epsilon^t$ by Lemma~\ref{lemma:LD-prep2}. Taking logarithms and dividing by the (negative!) quantity $\log\epsilon$ yields the lower bound.
\end{proof}

For the upper bound we need another preparatory lemma.

\begin{lemma}\label{lemma:LD-prep3}
For all $\beta>0$ and $\ell>0$ there exist $\gamma=\gamma(\beta)>0$ 
and $n_0=n_0(\ell)\in\N$
such that for all $z\in\J$, each interval $I\subseteq\I$ of length at least $\ell$ and all $n\geqslant n_0$ holds:
\begin{equation*}
m\left\{{\omega}\in I: \varphi_c({\omega})< z\right\}
\geqslant
\gamma\cdot m\left\{\omega\in \tilde I: f_\omega^n(z)-\varphi^-\geqslant\beta\right\},
\end{equation*}
where $\tilde{I}$ denotes the middle third of $I$.
Indeed, one can choose $\gamma:=e^{-D}\cdot\min_{K\in\cK} m\{\omega\in K:\varphi_c(\omega)<\varphi^-+\beta e^{-D}\}$, where the minimum extends
over the family $\cK$ of all Markov intervals of $\S$.
\end{lemma}
\begin{proof}
Fix $z\in\J$, $n\geqslant1$ and an interval $I\subset\I$, and
denote by $\cU$ the family of all maximal monotonicity intervals 
$U\subseteq I$ of $\S^n$ which contain a point $\omega_U$ such that $f_{\omega_U}^n(z)-f_{\omega_U}^n(\varphi^-)=f_{\omega_U}^n(z)-\varphi^-\geqslant\beta$. (Such an interval need not exist for each $\beta>0$.)
Denote for the moment the inverse of $\S^n|_U$ by $\rho=\rho_U:S^n(U)\to U$.
Then, for each $U\in\cU$,
\begin{equation*}
\begin{split}
&
m\left\{\omega\in \S^n(U):\varphi_c(\omega)-\varphi^-<\beta e^{-D}\right\}\\
\leqslant&
m\left\{\omega\in \S^n(U):\varphi_c(\omega)-\varphi^-< (f_{\omega_U}^n(z)-f_{\omega_U}^n(\varphi^-))\, e^{-D}\right\}\\
\leqslant&
m\left\{\omega\in \S^n(U):f_{\rho(\omega)}^n(\varphi_c(\rho(\omega)))-f_{\rho(\omega)}^n(\varphi^-)< f_{\rho(\omega)}^n(z)-f_{\rho(\omega)}^n(\varphi^-)\right\}\\
=&
m\left\{\omega\in \S^n(U):\varphi_c(\rho(\omega))< z\right\},
\end{split}
\end{equation*}
where we used (\ref{eq:distortion3}) for the second inequality and the monotonicity of $f_{\rho(\omega)}^n$ for the last one. In view of the distortion bound (\ref{eq:distortion1}) this implies
\begin{equation*}
\frac{1}{m(S^n(U))}\, m\left\{\omega\in \S^n(U):\varphi_c(\omega)-\varphi^-<\beta e^{-D}\right\}
\leqslant
e^D\cdot \frac{1}{m(U)}\, m\left\{\omega\in U: \varphi_c(\omega)< z\right\}.
\end{equation*}
Therefore
\begin{equation*}
\begin{split}
\sum_{U\in\cU} m\left\{\omega\in U: \varphi_c(\omega)< z\right\}
&\geqslant
e^{-D}\cdot \sum_{U\in\cU} \frac{m(U)}{m(\S^n(U))}\cdot m\left\{\omega\in \S^n(U):\varphi_c(\omega)-\varphi^-<\beta e^{-D}\right\}\\
&\geqslant
\gamma\cdot\sum_{U\in\cU}m(U)\ .
\end{split}
\end{equation*}
For the last inequality we used that $S^n(U)$ contains at least one Markov interval when $n\geqslant n_0$ and $n_0$ is sufficiently large. Choosing $n_0=n_0(\ell)$ even larger, if necessary, 
the at most two $U\in\cU$ which are not fully contained in $I$ are disjoint to $\tilde{I}$.
Hence
\begin{equation*}
m\left\{\omega\in I: \varphi_c(\omega)< z\right\}
\geqslant
\gamma\cdot\sum_{U\in\cU}m(U)
\geqslant
\gamma\cdot m\left\{\omega\in \tilde I: f_\omega^n(z)-\varphi^-\geqslant\beta \right\}.
\end{equation*}

It remains to prove that $\gamma$ is strictly positive: 
Let $ \vartheta:=e^{-D}\beta$ and suppose for a contradiction that there is some Markov interval $K$ of $\S$ such that $\varphi_c(\omega)\geqslant\varphi^-+\vartheta$ for all $\omega$ in a full measure subset $K_0$ of $K$.
As $\S$ is a mixing Markov map, there is $k\in\N$ such that $m(\I\setminus\S^k(K_0))=0$. Hence, for $m$-a.e. $\tilde{\omega}\in\I$ there is $\omega'\in K_0$ such that $\varphi_c(\tilde{\omega})=f_{\omega'}^k(\varphi_c(\omega'))
\geqslant\inf_{\omega\in\I} f_\omega^k(\vartheta)>\varphi^-$, which is incompatible with 
Proposition~\ref{prop:intermingled}.
\end{proof}

In order to relate the expression $f_\omega^n(z)-\varphi^-$ to $(f_\omega^n)'(\varphi^-)$, we invoke a rather immediate consequence of the fact that maps with negative Schwarzian derivative satisfy a variant of the one-sided Koebe principle \cite[Section IV.1, Property 4 and Corollary 2]{dMvS}:
\begin{equation}\label{eq:os-Koebe}
\frac{f_\omega^n(z)-\varphi^-}{z-\varphi^-}\geqslant \frac{\varphi^+-f_\omega^n(z)}{\varphi^+-z}\cdot(f_\omega^n)'(\varphi^-)\quad
\text{for all $\omega\in\I$, $z\in\J$ and $n\geqslant1$.}
\end{equation}
Indeed, this follows most directly from the very transparent discussion in 
\cite[Appendix A]{Nowicki1998} by applying their formula (71) with $I=g(I):=\J$ in the limit where their interval $L$ shrinks to the point $-1$.

\begin{proof}[Proof of Theorem~\ref{theo:exponential},
upper bound: 
$\limsup_{\epsilon\to0}\frac{\log m\{\varphi_c<\varphi^-+\epsilon)\}}{\log \epsilon}\leqslant t_*^-$]\quad\\
For later use we prove a slightly stronger statement, namely:
\begin{equation}\label{eq:stronger-statement}
\parbox{13cm}{\centering
 If $(I_\epsilon)_{\epsilon>0}$ is any family intervals with $\ell:=\inf_\epsilon|I_\epsilon|>0$, then\\[2mm]
$\limsup_{\epsilon\to0}\frac{\log m\{\omega\in I_\epsilon:\,\varphi_c(\omega)<\varphi^-+\epsilon)\}}{\log \epsilon}\leqslant t_*^-$.}
\end{equation}

Fix any $\beta\in(0,1)$, e.g. $\beta=\frac{1}{2}$. If 
$\omega\in\I$, $z\in(-1,0)\subset\J$ and $n\geqslant1$ are such that $(f_\omega^n)'(\varphi^-)\geqslant\frac{2\beta}{z-\varphi^-}$, then
$f_\omega^n(z)-\varphi^-\geqslant\beta$: indeed, otherwise
$\varphi^+-f_\omega^n(z)>2-\beta$ so that (\ref{eq:os-Koebe}) implies
$(f_\omega^n)'(\varphi^-)<\frac{\beta}{2-\beta}\,\frac{\varphi^+-z}{z-\varphi^-}<\frac{2\beta}{z-\varphi^-}$. Hence Lemma~\ref{lemma:LD-prep3} yields for each $z=\varphi^-+\epsilon$ with $\epsilon\in(0,1)$ and $n\geqslant n_0(\ell)$
\begin{equation*}
m\left\{{\omega}\in I_\epsilon: \varphi_c({\omega})< z\right\}
\geqslant
\gamma\cdot\left\{\omega\in \tilde I_\epsilon: (f_\omega^n)'(\varphi^-)\geqslant\frac{2\beta}{z-\varphi^-}\right\},
\end{equation*}
so that, for any choice of $n_\epsilon\geqslant n_0(\ell)$,
\begin{equation*}
\begin{split}
\limsup_{\epsilon\to0}\frac{\log m\{\omega\in I_\epsilon:\varphi_c(\omega)<\varphi^-+\epsilon)\}}{\log \epsilon}
\leqslant
\limsup_{\epsilon\to0}\frac{\log m\{\omega\in \tilde I_\epsilon:(f_\omega^{n_\epsilon})'(\varphi^-)\geqslant\frac{2\beta}{\epsilon}\}}{\log \epsilon}.
\end{split}
\end{equation*}
Exactly as in \cite[Eqs.\,(4.9) and (4.10)]{Keller2012a},  the choice $n_\epsilon:=\lceil\alpha^{-1}|\log\epsilon|\rceil$ with $\alpha=(p^-)'(t_*^-)>0$ allows to identify this large deviations limit as $t_*^-$ (using the theorem of Plachky and Steinebach \cite{Plachky1975}).
\end{proof}

\subsection{Proof of Theorem~\ref{theo:main}}\label{subsec:proofTheo-main}

Let $\nu\in\cP_e(\S)$ be a Gibbs measure. Then  $\nu$-a.a. points
$\omega\in\I$ are regular in the following sense:
\begin{quote}
The limits
\begin{equation*}
\lim_{n\to\infty}\frac{1}{n}\log|(\S^n)'(\omega)|=\int\log|\S'|\,d\nu\ ,
\quad
\lim_{n\to\infty}\frac{1}{n}\log(f_\omega^n)'(\varphi^\pm)=\lambda_\nu(\varphi^\pm)
\end{equation*}
exist,
and there are sequences of integers $n_1<n_2<\dots$ and of reals $\epsilon_1>\epsilon_2>\dots\searrow0$ such that the symmetric $\epsilon_k$-neighbourhoods $V_{\epsilon_k}(\omega)$ of $\omega$ satisfy
\begin{equation*}
\S^{n_k}|_{V_{\epsilon_k}}:V_{\epsilon_k}(\omega)\to \S^{n_k}(V_{\epsilon_k}(\omega))\text{ is a diffeomorphism and }
\inf_k|\S^{n_k}(V_{\epsilon_k}(\omega))|>0,
\end{equation*}
and such that it is enough to evaluate 
\begin{equation}\label{eq:sigma-pm-rep}
\sigma^\pm(\omega,x)=\lim_{\epsilon\to0}\frac{1}{\log\epsilon}\log
\frac{m^2\left(U_\epsilon(\omega,x)\cap\basin^\pm\right)}{m^2\left(U_\epsilon(\omega,x)\right)}
\end{equation}
along the sequence $\epsilon_k\searrow0$, see \cite[Remark 5(a) and beginning of Section 5]{Keller2012a}. 
\end{quote}


In a first step we reduce the evaluation of the area quotient in (\ref{eq:sigma-pm-rep}) to the evaluation of quotients of one-dimensional measures.

\begin{lemma}\label{lemma:dim-reduction}
Suppose that $x<\varphi^+$. Then
\begin{equation*}
\begin{split}
\liminf_{k\to\infty}\frac{1}{\log\epsilon_k}\log
&\frac{m\left(V_{\epsilon_k}(\omega)\cap\{\varphi_c<x+\epsilon_k\}\right)}{m\left(V_{\epsilon_k}(\omega)\right)}\\
\leqslant
&\sigma^{+}(\omega,x)
\leqslant
\limsup_{k\to\infty}\frac{1}{\log\epsilon_k}\log
\frac{m\left(V_{\epsilon_k}(\omega)\cap\{\varphi_c<x+\epsilon_k/2\}\right)}{m\left(V_{\epsilon_k}(\omega)\right)}\ .
\end{split}
\end{equation*}
\end{lemma}

\begin{proof}
The lower bound follows from the inclusion $U_{\epsilon_k}(\omega,x)\cap\basin^+\subseteq(V_{\epsilon_k}(\omega)\cap\{\varphi_c<x+\epsilon_k\})
\times[x-\epsilon_k,x+\epsilon_k]$, the upper bound from $U_{\epsilon_k}(\omega,x)\cap\basin^+\supseteq(V_{\epsilon_k}(\omega)
\cap\{\varphi_c<x+\epsilon_k/2\})\times(\J\cap[x+\epsilon_k/2,x+\epsilon_k])$,
which holds for $x<\varphi^+$ and sufficiently small $\epsilon_k>0$. 
\end{proof}
In view of the distortion bound (\ref{eq:distortion1}) for $\S$ we have
\begin{equation}\label{eq:area-distortion}
e^{-D}
\leqslant
\frac{m\left(V_{\epsilon_k}(\omega)\cap\{\varphi_c<x+\epsilon_k/2\}\right)}{m\left(V_{\epsilon_k}(\omega)\right)}\Bigg/ 
\frac{m\left(\S^{n_k}(V_{\epsilon_k}(\omega)\cap\{\varphi_c<x+\epsilon_k/2\})\right)}
{m\left(\S^{n_k}(V_{\epsilon_k}(\omega))\right)}
\leqslant
e^{D}.
\end{equation}
Next observe that $\S^{n_k}(V_{\epsilon_k}(\omega)\cap\{\varphi_c<x+\epsilon_k/2\})
=\S^{n_k}(V_{\epsilon_k}(\omega))\cap\{\varphi_c<h_k\}$, where $h_k:\S^{n_k}(V_{\epsilon_k}(\omega))\to\J$ satisfies $h_k(\S^{n_k}(\tilde{\omega}))=f_{\tilde{\omega}}^{n_k}(x+\epsilon_k/2)$ as in Lemma~\ref{lemma:dist-prep}.  There it is proved that 
\begin{equation*}
e^{-D}\cdot(h_k(\S^{n_k}(\omega))-\varphi^-)
\leqslant
h_k(z)-\varphi^-
\leqslant
e^{D}\cdot(h_k(\S^{n_k}(\omega))-\varphi^-)\quad\forall z\in\S^{n_k}(V_{\epsilon_k}(\omega))\ .
\end{equation*}
Therefore, observing also that $\inf_k m(\S^ {n_k}(V_{\epsilon_k}(\omega)))>0$,
we have for $x<\varphi^+$
\begin{equation}\label{eq:sigma+upper}
\begin{split}
\sigma^+(\omega,x)
&\leqslant\limsup_{k\to\infty}\frac{1}{\log\epsilon_k}\log
\frac{m\left(V_{\epsilon_k}(\omega)\cap\{\varphi_c<x+\epsilon_k/2\}\right)}{m\left(V_{\epsilon_k}(\omega)\right)}\\
&\leqslant
\limsup_{k\to\infty}\frac{1}{\log\epsilon_k}\log
{m\left(\S^{n_k}(V_{\epsilon_k}(\omega))\cap\{\varphi_c-\varphi^-<e^{-D}(f_\omega^{n_k}(x+\epsilon_k/2)-\varphi^-)\}\right)}
\\
&=
\begin{cases}
 t_*^-\cdot
\limsup_{k\to\infty}\frac{\log(f_\omega^{n_k}(x+\epsilon_k/2)-\varphi^-)}{\log\epsilon_k/2}&
\text{ if }\liminf_{k\to\infty}f_\omega^{n_k}(x+\epsilon_k/2)-\varphi^-=0\\
0&
\text{ if }\liminf_{k\to\infty}f_\omega^{n_k}(x+\epsilon_k/2)-\varphi^->0\ .
\end{cases}
\end{split}
\end{equation}
where we used Theorem~\ref{theo:exponential} (in its strengthened form (\ref{eq:stronger-statement})) for the equality. In a similar way one derives the lower bound
\begin{equation}\label{eq:sigma+lower}
\begin{split}
\sigma^+(\omega,x)
&\geqslant
\begin{cases}
{t_*^-}\cdot
\liminf_{k\to\infty}\frac{\log(f_\omega^{n_k}(x+\epsilon_k)-\varphi^-)}{\log\epsilon_{k}}&
\text{ if }\limsup_{k\to\infty}f_\omega^{n_k}(x+\epsilon_k)-\varphi^-=0\\
0&
\text{ if }\limsup_{k\to\infty}f_\omega^{n_k}(x+\epsilon_k)-\varphi^->0\ .
\end{cases}
\end{split}
\end{equation}
and we note that corresponding statements hold for $\sigma^-(\omega,x)$.

To proceed we denote $\gamma:=-\lim_{k\to\infty}\frac{\log\epsilon_{k}}{n_k}=-\lim_{k\to\infty}\frac{\log\epsilon_{k}/2}{n_k}=\int\log|\S'|\,d\P>0$.
We distinguish the following cases to prove Theorem~\ref{theo:main}, neglecting a set of $\omega$'s of $\nu$-measure $0$:
\begin{enumerate}[(A)]
\item{$x=\varphi_c(\omega)\in(\varphi^-,\varphi^+)$:}\\
Then $f_\omega^{n_k}(x+\epsilon_k/2)\geqslant f_\omega^{n_k}(x)=\varphi_c(\S^{n_k}(\omega))$
and 
\begin{equation*}
\begin{split}
\mac\left\{\omega:\varphi_c(\S^{n_k}(\omega))-\varphi^-<n_k^{-4/t_*^-}\right\}
&=
\mac\left\{\omega:\varphi_c(\omega)-\varphi^-<n_k^{-4/t_*^-}\right\}\\
&\leqslant
\left(n_k^{-4/t_*^-}\right)^{t_*^-/2}
=
n_k^{-2}
\end{split}
\end{equation*}
for sufficiently large $k$ by Theorem~\ref{theo:exponential}, so that the Borel Cantelli lemma yields
\begin{equation*}
\begin{split}
\limsup_{k\to\infty}\frac{\log(f_\omega^{n_k}(x+\epsilon_k/2)-\varphi^-)}{\log\epsilon_{k}/2}
&\leqslant
\limsup_{k\to\infty}\frac{\log(\varphi_c(\S^{n_k}(\omega))-\varphi^-)}{\log\epsilon_{k}/2}\\
&\leqslant
\limsup_{k\to\infty}\frac{\log(n_k^{-4/t_*^-})}{-\gamma\, n_k}
=0\ .
\end{split}
\end{equation*}
Hence $\lim_{k\to\infty}\frac{\log(f_\omega^{n_k}(x+\epsilon_k/2)-\varphi^-)}{\log\epsilon_{k}/2}
=0$, and
it follows from (\ref{eq:sigma+upper}) and (\ref{eq:sigma+lower}) 
that $\sigma^+(\omega,x)=0$. In the same way one proves $\sigma^-(\omega,x)=0$.
\item[(B$^-$)] {$\varphi^-<x<\varphi_c(\omega)$:}\\
In this case we can assume that $\varphi^-<\varphi_c$ $\nu$-a.e. , so that $\lambda_\nu(\varphi^-)<0$ by Proposition~\ref{prop:two-basins}. As $x<\varphi_c(\omega)$,
there is $k_0=k_0(\omega,x)$ such that $\lim_{n\to\infty}f_\omega^n(x+\epsilon_{n_{k_0}})-\varphi^-=0$. By monotonicity of the branches $f_\omega$ it follows that $\lim_{k\to\infty}f_\omega^{n_k}(x+p\epsilon_{n_{k}})-\varphi^-=0$ for each $p\in[0,1]$. So we can apply
Lemma~\ref{lemma:dist-folklore} and conclude that
\begin{equation*}
\lim_{k\to\infty}\frac{1}{n_k}\log(f_\omega^{n_k}(x+p\epsilon_{n_k})-\varphi^-)
=
\lim_{k\to\infty}\frac{1}{n_k}\log(f_\omega^{n_k})'(\varphi^-)
=
\lambda_\nu(\varphi^-)<0
\end{equation*}
for $\nu$-a.e. $\omega$ and $x\in(\varphi^-,\varphi_c(\omega))$. Referring to (\ref{eq:sigma+upper}) and (\ref{eq:sigma+lower}) this implies for such $(\omega,x)$
\begin{equation*}
\sigma^+(\omega,x)
=
t_*^-\cdot
\lim_{k\to\infty}\frac{\log(f_\omega^{n_k}(x+\epsilon_k)-\varphi^-)}{-\gamma\,n_k}
=t_*^-\cdot \frac{-\lambda_\nu(\varphi^-)}{\int\log|\S'|\,d\nu}>0\ ,
\end{equation*}
and it follows that $\sigma^-(\omega,x)=0$.
 
\item[(B$^+$)] {$\varphi_c(\omega)<x<\varphi^+$:}\\
As in (B$^-$) one proves 
$\sigma^-(\omega,x)
=
t_*^+\cdot \frac{-\lambda_\nu(\varphi^+)}{\int\log|\S'|\,d\nu}>0$ and
$\sigma^+(\omega,x)=0$ for $\nu$-a.e. $\omega$ and $x\in(\varphi_c(\omega),\varphi^+)$.
\item[(C$^-$)] $x=\varphi^-$ and $\lambda_\nu(\varphi^-)<\int\log|\S'|\,d\nu$: 
\\
This means that $\lambda_\nu(\varphi^-)-\lim_{k\to\infty}\frac{\log\epsilon_k}{n_k}<0$, so that one can apply Lemma~\ref{lemma:LD-prep1} with $\vartheta=p\epsilon_k$ for large $k$ to get
\begin{equation*}
\lim_{k\to\infty}\frac{1}{n_k}\log(f_\omega^{n_k}(x+p\epsilon_k)-\varphi^-)
=
\lambda_\nu(\varphi^-)-\int\log|\S'|\,d\nu<0
\end{equation*}
for $p=\frac{1}{2}$ and $p=1$,
so that
\begin{equation*}
\sigma^+(\omega,x)
=
t_*^-\cdot
\lim_{k\to\infty}\frac{\log(f_\omega^{n_k}(x+p\epsilon_k)-\varphi^-)}{-\gamma\,-n_k}
=t_*^-\cdot \frac{\int\log|\S'|\,d\nu\lambda_\nu(\varphi^-)}{\int\log|\S'|\,d\nu}>0\ ,
\end{equation*}
and hence $\sigma^-(\omega,x)=0$ for $\nu$-a.e. $\omega$ and $x=\varphi^-(\omega)$.
\item[(C$^+$)] $x=\varphi^+$ and $\lambda_\nu(\varphi^+)<\int\log|\S'|\,d\nu$.:
\\
As in (C$^-$) one proves 
$\sigma^-(\omega,x)
=t^+\cdot \frac{\int\log|\S'|\,d\nu-\lambda_\nu(\varphi^+)}{\int\log|\S'|\,d\nu}>0$ and $\sigma^+(\omega,x)=0$ for $\nu$-a.e. $\omega$ and $x=\varphi^+(\omega)$.
\end{enumerate}

\subsection{Proof of Theorem~\ref{theo:uncertainty}}\label{subsec:Proof_Theo-uncertainty}\label{subsec:proof-of-uncertainty}
As $S$ is an irreducible and aperiodic Markov map, there are $s,q\in\N$ such that $S^s(K)=\I$ (possibly except for the endpoints) for all Markov intervals $K\in\cK$, and such that each point in $\I$ has at most $q$ preimages under $S^s$.

For $u>0$ let 
\begin{equation}\label{eq:Phi-pm-def}
\Phi^-(u):=m\{\omega\in\I: \varphi_c(\omega)<\varphi^-+u\}\text{ and }
\Phi^+(u):=m\{\omega\in\Omega:\varphi_c(\omega)>\varphi^+-u\}
\end{equation}

Denote by $\tilde n_\epsilon(\omega)$ 
the smallest integer such that $|(S^n)'(\omega)|>e^D\epsilon^{-1}$.
We denote the maximal monotonicity and continuity interval of $S^{\tilde n_\epsilon(\omega)}$ around $\omega$ by $Z_\epsilon(\omega)$. Observe that the distortion estimate (\ref{eq:distortion1}) guarantees that the length of each
$Z_\epsilon(\omega)$ is between $e^{-2D}\epsilon$ and $\epsilon$. In particular, $Z_\epsilon(\omega)\subseteq B_\epsilon(\omega)$ where $B_\epsilon(\omega)$ denotes the $\epsilon$-neighbourhood of $\omega$ in $\I$. 
From each collection $(Z_\epsilon(\omega):\omega\in\I)$ one can extract a Moran cover $\cU_\epsilon$ of $\I$, that is a disjoint (modulo endpoints) collection of sets $Z_\epsilon(\omega)$ covering all of $\I$ (modulo endpoints), see e.g. \cite[Section 13]{pesin-book}. By construction, each $Z\in\cU_\epsilon$ is a maximal monotonicity interval of some $S^{n_Z}$.

With these conventions we have
\begin{equation*}
\begin{split}
&\Pi_{\epsilon,x}\\
&=
(2\epsilon)^{-1}\cdot m^2\left\{(\omega,\omega')\in\I^2: \varphi_c(\omega)>x, \varphi_c(\omega')<x, |\omega'-\omega|<\epsilon\right\}\\
&\geqslant
(2\epsilon)^{-1}\cdot \sum_{Z\in\cU_\epsilon}
m^2\left\{(\omega,\omega')\in Z^2: \varphi_c(\omega)>x, \varphi_c(\omega')<x\right\}\\
&=
(2\epsilon)^{-1}\cdot \sum_{Z\in\cU_\epsilon}
m\left\{\omega\in Z: \varphi_c(\omega)>x\right\}
\cdot m\left\{\omega\in Z: \varphi_c(\omega)<x\right\}\\
\end{split}
\end{equation*}
Because of Lemma~\ref{lemma:dist-prep}, there is a constant $\kappa>0$ such that
for each $Z=Z_\epsilon(\tilde{\omega})\in\cU_\epsilon$ holds
\begin{equation}
\begin{split}
&m\left\{\omega\in Z: \varphi_c(\omega)>x\right\}\\
=&
m\left\{\omega\in Z: \varphi_c(S^{n_Z+s}(\omega))>f_\omega^{n_Z+s}(x)\right\}\\
\geqslant&
m\left\{\omega\in Z: \varphi^+-\varphi_c(S^{n_Z+s}(\omega))<\kappa\cdot(\varphi^+- f_{\tilde\omega}^{n_Z+s}(x))\right\}\label{eq:kappa-first}\\
\geqslant&
q^{-1}e^{-D}\cdot m(Z)\cdot 
m\left\{\omega\in\I: \varphi^+-\varphi_c(\omega)<\kappa\cdot(\varphi^+- f_{\tilde\omega}^{n_Z+s}(x))\right\}\\
=&
q^{-1}e^{-D}\cdot m(Z)\cdot\Phi^+\left(\kappa\cdot(\varphi^+- f_{\tilde\omega}^{n_Z+s}(x))\right)
\end{split}
\end{equation}
and similarly
\begin{equation*}
m\left\{\omega\in Z: \varphi_c(\omega)<x\right\}
\geqslant
q^{-1}e^{-D}\cdot m(Z)\cdot\Phi^-\left(\kappa\cdot( f_{\tilde\omega}^{n_Z+s}(x)-\varphi^-)\right)
\end{equation*}
Observing that $m(Z)=2\epsilon$ for all $Z\in\cU_\epsilon$, we conclude
\begin{equation}\label{eq:first-estimate-for Pi}
\begin{split}
\Pi_{\epsilon,x}
&\geqslant
 q^{-2}e^{-2D}
 \sum_{Z=Z_\epsilon(\tilde\omega)\in\cU_\epsilon}m(Z)\cdot
 \Phi^+\left(\kappa\cdot(\varphi^+-f_{\tilde\omega}^{n_Z+s}(x))\right)\cdot\Phi^-\left(\kappa\cdot( f_{\tilde\omega}^{n_Z+s}(x)-\varphi^-)\right)\\
 &\geqslant
q^{-2}e^{-2D}
\int_\I \Phi^+\left(\kappa^2\cdot(\varphi^+-f_\omega^{n_\epsilon(\omega)}(x))\right)\cdot\Phi^-\left(\kappa^2\cdot( f_\omega^{n_\epsilon(\omega)}(x)-\varphi^-)\right)\,dm(\omega)\ ,
\end{split}
\end{equation}
where ${n}_\epsilon(\omega)=n_Z+s$ when $\omega\in Z\in\cU_\epsilon$.

In the next step we compare $(f_\omega^{n_\epsilon(\omega)}(x)-\varphi^-)$ to 
$(f_\omega^{n_\epsilon(\omega)})'(\varphi^-)$:
As all branches $f_\omega$ have negative Schwarzian derivative, we have
\begin{equation*}
\frac{{f_\omega^n(x)}-\varphi^-}{x-\varphi^-}\cdot\frac{|\J|}{|\J|}
\geqslant
(f_\omega^n)'(\varphi^-)\cdot\frac{\varphi^+-{f_\omega^n(x)}}{\varphi^+-x},
\end{equation*}
so that
\begin{equation*}
{f_\omega^n(x)}-\varphi^-
\geqslant
f_\omega^n)'(\varphi^-)\cdot\frac{x-\varphi^-}{\varphi^+-x}\cdot{(\varphi^+-{f_\omega^n(x)})}.
\end{equation*}
As $\varphi^+-\varphi^-=2$, it follows that
\begin{equation*}
f_\omega^n(x)-\varphi^-
\geqslant
\min\left\{1,(f_\omega^n)'(\varphi^-)\cdot\frac{x-\varphi^-}{\varphi^+-x}\right\}.
\end{equation*}
Let $c=|\log\frac{x-\varphi^-}{\varphi^+-x}|$ (recall that $x$ is fixed). Then
\begin{equation*}
\log\left(\kappa^2\cdot({f_\omega^n(x)}-\varphi^-)\right)
\geqslant
\min\left\{0,\log(f_\omega^n)'(\varphi^-)-c\right\}+\log\kappa^2.
\end{equation*}
For $z\in\R$ and $\epsilon>0$ define 
$\zeta_\epsilon(z)=\min\left\{0,z-\frac{c}{|\log\epsilon|}\right\}+\frac{\log\kappa^2}{|\log\epsilon|}$. Then the last inequality can be rewritten as
\begin{equation*}
\frac{1}{|\log\epsilon|}
\log\left(\kappa^2\cdot(f_\omega^n(x)-\varphi^-)\right)
\geqslant
\zeta_\epsilon\left(\frac{1}{|\log\epsilon|}\log(f_\omega^n)'(\varphi^-)\right),
\end{equation*}
so that
\begin{equation}\label{eq:comparison-minus}
\kappa^2\cdot(f_\omega^n(x)-\varphi^-)
\geqslant
\epsilon^{-\zeta_\epsilon\left(|\log\epsilon|^{-1}\log(f_\omega^n)'(\varphi^-)\right)}.
\end{equation}
Similarly,
\begin{equation}\label{eq:comparison-plus}
\kappa^2\cdot(\varphi^+-{f_\omega^n(x)})
\geqslant
\epsilon^{-\zeta_\epsilon\left(|\log\epsilon|^{-1}\log(f_\omega^n)'(\varphi^+)\right)}.
\end{equation}
Therefore, observing (\ref{eq:first-estimate-for Pi}) and the monotonicuty of $\Phi^\pm$,
\begin{equation}\label{eq:intermediate-phi}
\Pi_{\epsilon,x}
\geqslant
q^{-2}e^{-2D}
\int_\I\Phi^-
\left(\epsilon^{-\zeta_\epsilon\left(|\log\epsilon|^{-1}\log(f_\omega^n)'(\varphi^-)\right)}\right)
\cdot
\Phi^+\left(
\epsilon^{-\zeta_\epsilon\left(|\log\epsilon|^{-1}\log(f_\omega^n)'(\varphi^+)\right)}
\right)\,dm(\omega)\ .
\end{equation}

We determine the asymptotics of this integral when $\epsilon\to0$ using  Varadhan's integral lemma. This needs some preparation: Consider
the three-variate process
\begin{equation}\label{eq:Yn-def}
Y_n
=
\begin{pmatrix}
Y_n^-\\Y_n^+\\Y_n^S
\end{pmatrix}
:=
\begin{pmatrix}
n^{-1}\log (f_\omega^n)'(\varphi^-)\\
n^{-1}\log (f_\omega^n)'(\varphi^+)\\
n^{-1}\log |(S^n)'(\omega)|
\end{pmatrix}.
\end{equation}
We will determine the large deviations of this process and related ones.

As $\log f_\omega'(\varphi^\pm)$ and $\log|S'(\omega)|$ are H\"older
continuous on each set $I_i\times\J$, one can use transfer operators to prove the existence, convexity and smoothness
of the asymptotic log-Laplace transform
\begin{equation}\label{eq:psi-def}
\psi:\R^3\to\R,\quad\psi(\lambda):=\lim_{n\to\infty}\frac{1}{n}
\log\int_\I\exp\left(n\langle \lambda,Y_n\rangle\right)\,dm\ .
\end{equation}
Indeed, $\psi$ coincides with the pressure function (\ref{eq:pressure-function}) defined in Theorem~\ref{theo:uncertainty}.\footnote{This is a well established theory, see e.g. \cite[Section 7]{keller2014} for some details and references.}

Here are some particular values of $\psi$ that help to understand its geometry:
\begin{eqnarray}
&
 \psi(0,0,0)=\psi(t_*^-,0,0)=\psi(0,t_*^+,0)=0
\label{eq:values-1}\\
&\frac{\partial\psi}{\partial \lambda^\pm}(0,0,0)=\lambda_{\mac}(\varphi^\pm)<0,\quad
\frac{\partial\psi}{\partial \lambda^-}(t_*^-,0,0)>0,
\quad
\frac{\partial\psi}{\partial \lambda^+}(0,t_*^+,0)>0
\label{eq:values-2}\\
&0<\inf_{(\lambda^-,\lambda^+,\lambda^S)\in\R^3}\frac{\partial\psi}{\partial \lambda^S}(\lambda^-,\lambda^+,\lambda^S)
\leqslant \sup_{(\lambda^-,\lambda^+,\lambda^S)\in\R^3}\frac{\partial\psi}{\partial \lambda^S}(\lambda^-,\lambda^+,\lambda^S)<\infty\label{eq:values-3}
\end{eqnarray}
Negative Schwarzian derivative of the branches implies furthermore:
\begin{equation*}
\sup_{(\lambda^-,\lambda^+,\lambda^S)\in\R^3}\left(\frac{\partial \psi}{\partial \lambda^-}(\lambda^-,\lambda^+,\lambda^S)+\frac{\partial \psi}{\partial \lambda^+}(\lambda^-,\lambda^+,\lambda^S)\right)<0\ .
\end{equation*}
This is a direct consequence of Propossition~\ref{prop:JBM}, because the partial derivatives are precisely the exponents for the equilibrium state associated with $(\lambda^-,\lambda^+,\lambda^S)$.
In particular, $\psi$ is unbounded below and above.

Let $\psi^*(x):=\sup_{\lambda\in\R^3}
\left(\langle\lambda,x\rangle-\psi(\lambda)\right)$ be the Legendre-Fenchel transform of $\psi$. The G\"artner-Ellis theorem guarantees that the large deviations principle with rate $n$ and with good convex rate function $\psi^*$ holds for the sequence $(Y_n)_{n\in\N}$ \cite[Theorem 2.3.6(c)]{Dembo-Zeitouni-book}. Even more holds in the present situation:
\begin{equation}\label{eq:LDP-A-1}
\parbox{14cm}{
For $t\in\R$ denote $Y_n(t):=t\,Y_{[nt]}$.
Fix $m\in\N$ and $0=t_0<t_1<\dots<t_m\leqslant1$. Then the process
$\tilde{Y}_n:=\left(Y_n(t_1),Y_n(t_2)-Y_n(t_1),\dots,Y_n(t_m)-Y_n(t_{m-1})\right)$
satisfies the LDP in $(\R^3)^m$ with good rate function
\begin{equation*}
\psi_m^*:(\R^3)^m\to[-\infty,\infty],\quad
\psi_m^*(z)=\sum_{i=1}^m(t_i-t_{i-1})\psi^*\left(\frac{z_i}{t_i-t_{i-1}}\right).
\end{equation*}
}
\end{equation}
This is exactly condition (A-1) from \cite{Dembo1995}. Observe that $\psi_m^*$ is the Legendre-Fenchel transform of $\psi_m(\lambda):=\sum_{i=1}^m(t_i-t_{i-1})\psi(\lambda_i)$. We show in Lemma~\ref{lemma:Laplace} at the end of this section
 that $\psi_m$ is the asymptotic log-Laplace transform of $(\tilde{Y}_n)_{n\in\N}$ and that it is convex and smooth. Then the G\"artner-Ellis theorem guarantees the LDP as claimed in (\ref{eq:LDP-A-1}).

As $0<\inf_{\omega\in\I}\log|S'(\omega)|\leqslant\sup_{\omega\in\I}\log|S'(\omega)|<\infty$, one can apply, as a next step, the time
change techniques from \cite[Sections 2-4]{Duffy2004}:
Theorem 1 from \cite{Dembo1995} together with Theorem 1 from \cite{Ganesh2002}
allow to apply Theorem 4 (and the subsequent further discussion) from \cite{Duffy2004} and to conclude \footnote{Formally only the case of  bivariate processes $(Y_n)_{n\in\N}$ is treated in \cite{Duffy2004}. A look at that paper, hovever, reveals that the results from \cite{Dembo1995} and \cite{Ganesh2002} referred to in \cite{Duffy2004} apply to our situation as well and that the proof of Theorem 4 in \cite{Duffy2004} carries over without any changes.}
\footnote{A more attentive look at \cite{Duffy2004} reveals that the results from that paper apply immediately only to the process
$(\tilde n_\epsilon/|\log\epsilon|\cdot Y_{n_\epsilon})_{\epsilon>0}$. However, the distortion property (\ref{eq:distortion1}) and the uniform expansion of $S$ immediately imply that the difference $\tilde{n}_\epsilon(\omega)-n_\epsilon(\omega)$ is bounded uniformly in
$\epsilon$ and in $\omega$. In particular, both processes are exponentially equivalent so that one can be replaced by the other without changing the large deviations behaviour. 
A related question is also discussed
after Theorem 2 of \cite{Duffy2004}.}
:
\begin{equation*}
\parbox{14.5cm}{
The process $(n_\epsilon/|\log\epsilon|\cdot Y_{n_\epsilon})_{\epsilon>0}$ satisfies the large deviation principle with rate $|\log \epsilon|$ and good convex rate function
\begin{equation}
J:\R^2\times(0,\infty)\to[-\infty,\infty],\quad J(x^-,x^+,y)=y\cdot \psi^*\left(\frac{x^-}{y},\frac{x^+}{y},\frac{1}{y}\right).
\end{equation}
}
\end{equation*}

Next we apply Varadhan's lemma \cite[Theorem 4.3.1]{Dembo-Zeitouni-book} to the process $(n_\epsilon/|\log\epsilon|\cdot Y_{n_\epsilon})$ and the continuous function 
\begin{equation*}
H:\R^2\times(0,\infty)\to\R,\quad H(x^-,x^+,y):=\min\{t_*^-x^-,t_*^+x^+,0\}
\end{equation*}
in order to evaluate 
\begin{equation}\label{eq:Varadhan}
\phi:=-\lim_{\epsilon\to0}\frac{1}{|\log\epsilon|}\log\int_\I \exp\left(|\log \epsilon|\cdot H(n_\epsilon/|\log\epsilon|\cdot Y_{n_\epsilon(\omega)}(\omega))\right)dm(\omega)\ ,
\end{equation}
namely:
\begin{equation}\label{eq:phi1}
\begin{split}
\phi
&=
-\sup_{(x^-,x^+)\in\R^2}\sup_{y>0}\left(H(x^-,x^+,y)-J(x^-,x^+,y)\right)\\
&=
\inf_{(x^-,x^+)\in\R^2}\inf_{y>0}\left(y\cdot\psi^*\left(\frac{x^-}{y},\frac{x^+}{y},\frac{1}{y}\right)-H(x^-,x^+,y)
\right)\\
&=
\inf_{(x^-,x^+)\in\R^2}\inf_{y>0}\sup_{\lambda\in\R^3}\left(
\langle \lambda,(x^-,x^+,1)\rangle-y\cdot\psi(\lambda)-H(x^-,x^+,y)
\right).
\end{split}
\end{equation}
The following lemma, whose proof is deferred to the end of this section, relates this to the estimate of
(\ref{eq:intermediate-phi}) we are looking for:
\begin{lemma}\label{lemma:Phi-H}
Let $c,\kappa>0$. 
Then for each $C>0$,
\begin{equation*}
\lim_{\epsilon\searrow0}\;\sup_{z^-,z^+\geqslant-C}\;\sup_{y\in\R}
\left|\frac{1}{|\log\epsilon|}
\log\left(\Phi^-(\epsilon^{-\zeta_\epsilon(z^-)})\cdot
\Phi^+(\epsilon^{-\zeta_\epsilon(z^+)})\right)-H(z^-,z^+,y)
\right|
=0\ .
\end{equation*}
\end{lemma}

As $\log f_\omega'(x)$ and $\log|S'(\omega)|$ are uniformly bounded in $\omega$ and $x$, the process $|\log \epsilon|\cdot H(n_\epsilon/|\log\epsilon|\cdot Y_{n_\epsilon(\omega)}(\omega))$ is uniformly bounded, so that (\ref{eq:Varadhan}) together with Lemma~\ref{lemma:Phi-H} implies:
\begin{equation}\label{eq:LD-limit}
\begin{split}
\phi
&=
-\lim_{\epsilon\to0}\frac{1}{|\log\epsilon|}\log\int_\I \exp\left(|\log \epsilon|\cdot H(n_\epsilon/|\log\epsilon|\cdot Y_{n_\epsilon(\omega)}(\omega))\right)dm(\omega)\\
&=
-\lim_{\epsilon\to0}\frac{1}{|\log\epsilon|}\log\int_\I 
\Phi^-(\epsilon^{-\zeta_\epsilon(|\log\epsilon|^{-1}
\log(f_\omega^{n_\epsilon(\omega)})'(\varphi^-))})\cdot
\Phi^+(\epsilon^{-\zeta_\epsilon(|\log\epsilon|^{-1}
\log(f_\omega^{n_\epsilon(\omega)})'(\varphi^+))})\,
dm(\omega)
\end{split}
\end{equation}
%

It remains to evaluate the expression (\ref{eq:phi1}) for $\phi$. To that end we introduce some more notation. Define $\ell:\R\to\R^2$,
\begin{equation*}
\ell(u)
=
\frac{1}{2}
\left(
(1-u)t_*^-,
(1+u)t_*^+
\right)
\end{equation*}
so that $\ell(-1)=(t_*^-,0)$ and $\ell(1)=(0,t_*^+)$. Then $\tilde{\psi}:\R\to\R$, $\tilde{\psi}(u):=\psi(\ell(u),0)$, is convex with $\tilde{\psi}(-1)=\psi(t_*^-,0,0)=0$ and
$\tilde{\psi}(1)=\psi(0,t_*^+,0)=0$ by (\ref{eq:values-1}). Therefore (\ref{eq:values-3}) implies that there are uniquely determined $h(u)\in\R$ such that $\psi(\ell(u),h(u))=0$ for all $u\in\R$, $h(\pm1)=0$, and $h(u)>0$ if and only if $u\in(-1,1)$.  Observe also that $h$ is real analytic by the implicit function theorem, because $\psi$ and $\ell$ are.

We claim that $u\mapsto h(u)$ has a unique maximum (that is attained inside the interval  $(-1,1)$, of course). Indeed: 
\begin{equation*}
\langle D\psi(\ell(u),h(u)),(\ell'(u),h'(u))\rangle=0
\end{equation*}
so that
\begin{equation*}
(\ell'(u),h'(u))\cdot D^2\psi(\ell(u),h(u))\cdot
\begin{pmatrix}
\ell'(u)\\h'(u)
\end{pmatrix}
+
\langle D\psi(\ell(u),h(u)),(\ell''(u),h''(u))\rangle=0\ .
\end{equation*}
As  $\ell''(u)=0$, as $\frac{\partial\psi}{\partial\lambda^S}>0$ and as $\psi$ is convex, this implies
\begin{equation*}
h''(u)
=
-\left(\frac{\partial\psi}{\partial\lambda^S}(\ell(u),h(u))\right)^{-1}
\cdot
(\ell'(u),h'(u))\cdot D^2\psi(\ell(u),h(u))\cdot
\begin{pmatrix}
\ell'(u)\\h'(u)
\end{pmatrix}
\leqslant 0
\end{equation*}
As we cannot guarantee strict concavity of $\psi$, this inequality might possibly be non-strict. But assume that there are $u_1<u_2$ such that $h(u)=c$ is constant for $u_1<u<u_2$. 
As $h(u)$ is analytic in $u\in\R$, this implies $h(u)=c$ for all $u\in\R$. But this contradicts the sign change of $h$ at $u=\pm1$ noted above.

Denote by $u_*$ the unique $u$ where $h(u)$ attains its maximum. Obviously $u_*\in(-1,1)$, $h(u_*)>0$ and $h'(u_*)=0$. Let $\lambda_*:=(\ell(u_*),h(u_*))$, $y_*:=\left(\frac{\partial\psi}{\partial\lambda^S}(\lambda_*)\right)^{-1}>0$ and 
$x^-_*:=y_*\,\frac{\partial\psi}{\partial\lambda^-}(\lambda_*)$,
$x^+_*:=y_*\,\frac{\partial\psi}{\partial\lambda^+}(\lambda_*)$.
Then $\psi(\lambda_*)=0$, $y_*D\psi(\lambda_*)=(x^-_*,x^+_*,1)$, and
\begin{equation}\label{eq:balance}
-t_*^-x^-_*+t_*^+x^+_*
=
2y_*\cdot\frac{\partial}{\partial u}\psi(\ell(u),h(u))_{|u=u_*}
-2y_*\cdot\frac{\partial\psi}{\partial\lambda^S}(\lambda_*)\,h'(u_*)
=0\ ,
\end{equation}
because $\psi(\ell(u),h(u))=0$ for all $u$ and $h'(u_*)=0$. In particular, $x^-_*$ and $x^+_*$ have the same sign, and as 
$x_*^-/y_*$ and $x_*^+/y_*$ 
they are the exponents $\lambda_{\nu_*}(\varphi^-)$ and $\lambda_{\nu_*}(\varphi^+)$ for some equilibrium state $\nu_*$, they must both be negative, see Proposition~\ref{prop:two-basins} for details.

Now we are ready to prove that 
\begin{equation}\label{eq:phi}
\phi=h(u_*)\ . 
\end{equation}
In view of (\ref{eq:phi1}),
\begin{equation*}
\begin{split}
\phi
&\geqslant
\inf_{(x^-,x^+)\in\R^2}\inf_{y>0}\left(
\langle \lambda_*,(x^-,x^+,1)\rangle-y\cdot\psi(\lambda_*)-H(x^-,x^+,y)
\right)\\
&=
\inf_{(x^-,x^+)\in\R^2}\inf_{y>0}\left(
\langle \lambda_*,(x^-,x^+,1)\rangle-H(x^-,x^+,y)
\right)\\
&=
\inf_{(x^-,x^+)\in\R^2}\left(\frac{1-u_*}2t_*^-x^-+\frac{1+u_*}2t_*^+x^++h(u_*)
-\min\{t_*^-x^-,t_*^+x^+,0\}\right)\\
&=h(u_*)\ .
\end{split}
\end{equation*}
For the upper estimate recall that $z_*:=t_*^-x^-_*=t_*^+x^+_*<0$ from (\ref{eq:balance}) and the lines thereafter. Hence
\begin{equation*}
\begin{split}
\phi
&\leqslant
\sup_{\lambda\in\R^3}\left(
\langle \lambda,(x^-_*,x^+_*,1)\rangle-y_*\cdot\psi(\lambda)-H(x^-_*,x^+_*,y_*)
\right)\\
&=
\sup_{\lambda\in\R^3}\left(
\langle \lambda-\lambda_*,(x^-_*,x^+_*,1)\rangle-y_*\cdot\left(\psi(\lambda)-\psi(\lambda_*)\right)
+\langle\lambda_*,(x^-_*,x^+_*,1)\rangle-z_*
\right)\\
&=
\sup_{\lambda\in\R^3}\left(
-y_*\cdot\left((\psi(\lambda)-\psi(\lambda_*))-
\langle\lambda-\lambda_*,D\psi(\lambda_*)\rangle\right)+\langle\lambda_*, (x^-_*,x^+_*,1)\rangle-z_*\right)\\
&=
\frac{1-u_*}{2}t_*^-x^-_*+\frac{1+u_*}{2}t_*^-x^+_*+h(u_*)-z_*
=
h(u_*)\ ,
\end{split}
\end{equation*}
where we used the convexity of $\psi$ for the last equality.
This finishes the proof of Theorem~\ref{theo:uncertainty}.
\qed

\quad\\
\begin{proof}[Proof of Lemma~\ref{lemma:Phi-H}]
Let $H^-(z):=\min\{t_*^-z,0\}$ and $\Delta^-_\epsilon(z):=\left|\frac{1}{|\log\epsilon|}
\log\Phi^-(\epsilon^{-\zeta_\epsilon(z)})-H^-(z)
\right|$.
For $t>c$ let $\Gamma(t):=\sup_{u\geqslant t}\left|\frac{\log \Phi^-(e^{-u})}{-u+c-\log\kappa}-t_*^-\right|$.
A straightforward short calculation 
based on Theorem~\ref{theo:exponential} yields that $\sup_{t\geqslant|c|}\Gamma(t)<\infty$ and $\lim_{t\to\infty}\Gamma(t)=0$. 
We prove that 
\begin{equation}\label{eq:Phi-1}
\lim_{\epsilon\searrow0}\sup_{z\geqslant-C}\Delta^-_\epsilon(z)=0
\end{equation}
for each $C>0$.
Indeed:\\
If $z>\frac{c}{|\log\epsilon|}$, then 
$\zeta_\epsilon(z)=\frac{\log\kappa}{|\log\epsilon|}$ and
$H^-(z)=0$, so that $\epsilon^{-\zeta_\epsilon(z)}=\kappa$ and $\Delta_\epsilon^-(z)=\left|\frac{\log\Phi^-(\kappa)}{\log\epsilon}\right|$.\\
If $|z|\leqslant\frac{c}{|\log\epsilon|}$, then 
$\zeta_\epsilon(z)=z-\frac{c}{|\log\epsilon|}+\frac{\log\kappa}{|\log\epsilon|}\leqslant\frac{\log\kappa}{|\log\epsilon|}$
and $|H(z)|\leqslant\frac{ t_*^-c}{|\log\epsilon|}$, so that
$\epsilon^{-\zeta_\epsilon(z)}\geqslant\kappa$
and $\Delta_\epsilon^-(z)\leqslant
\left|\frac{\log\Phi^-(\kappa)}{\log\epsilon}\right|
+\frac{ t_*^-c}{|\log\epsilon|}$.\\
If $-C\leqslant z<-\frac{c}{|\log\epsilon|}$, then $H^-(z)=t_*^-z$ and
\begin{equation*}
\begin{split}
\Delta^-_\epsilon(z)
&=
|z|\cdot\left|\frac{\log\Phi^-(\epsilon^{-\zeta_\epsilon(z)})}{z|\log \epsilon|}-t_*^-\right|
=
|z|\cdot\left|\frac{\log\Phi^-(e^{-\zeta_\epsilon(z)\log\epsilon})}{-\zeta_\epsilon(z)\log\epsilon+c-\log\kappa}-t_*^-\right|
\\
&\leqslant
|z|\cdot \Gamma(\zeta_\epsilon(z)\,\log\epsilon)\\
\end{split}
\end{equation*}
As $\zeta_\epsilon(z)=z-\frac{c}{|\log\epsilon|}+\frac{\log\kappa}{|\log\epsilon|}\leqslant\frac{\log\kappa}{|\log\epsilon|}$, we have $|z|\leqslant |\zeta_\epsilon(z)|+\left|\frac{c-\log\kappa}{\log\epsilon}\right|$. Hence
\begin{equation*}
\Delta_\epsilon^-(z)
\leqslant
\min\left\{|z|,|\zeta_\epsilon(z)|\right\}\cdot \Gamma(\zeta_\epsilon(z)\,\log\epsilon)+\Gamma(0)\cdot\left|\frac{c-\log\kappa}{\log\epsilon}\right|.
\end{equation*}
If $|\zeta_\epsilon(z)|\leqslant|\log\epsilon|^{-1/2}$, the first summand of this estimate is bounded by $\Gamma(0)\cdot|\log\epsilon|^{-1/2}$. Otherwise it is bounded by $C\cdot\Gamma(|\log\epsilon|^{1/2}|)$.\\
In any case we have
\begin{equation*}
\Delta_\epsilon^-(z)
\leqslant
\const\cdot\left(\frac{1}{|\log\epsilon|}+\frac{1}{|\log\epsilon|^{1/2}}+\Gamma(|\log\epsilon|^{1/2})\right)
\end{equation*}
which tends to zero when $\epsilon\to0$. This is (\ref{eq:Phi-1}).

In the same way one proves that
\begin{equation}\label{eq:Phi-2}
\lim_{\epsilon\searrow 0}\sup_{z\geqslant-C}
\left|\frac{1}{|\log\epsilon|}
\log\Phi^+(\epsilon^{-\zeta_\epsilon(z)})-H^+(z)
\right|=0\ ,
\end{equation}
where $H^+(z):=\min\{t^+z,0\}$. 
As $H^-(z^-)+H^+(z^+)=H(z^-,z^+,1)$, the lemma follows from (\ref{eq:Phi-1}) and (\ref{eq:Phi-2}).
\end{proof}

\begin{remark}\label{remark:Ott}
We compare the exponent $\phi$ with the corresponding one from \cite{Ott1994a,Ott1993} (also called $\phi$ in those papers), which is derived there for the case of riddled basins.
That formula is taken as a reference in \cite{Pereira2008}, although the authors of that paper study numerically the model from \cite{Hofbauer2004}, which has intermingled basins.
So we take the chance to derive an approximate formula in the style of  \cite{Ott1994a,Ott1993} for the model studied in \cite{Hofbauer2004,Pereira2008}, although it has the full logistic map as its driving system. But as that map is smoothly conjugated to the full symmetric tent map (with singularities only at the points $0$ and $1$), it is not hard to see that our Theorem~\ref{theo:uncertainty} carries over to this case. The fibre maps are $f_\omega(x)=x+\alpha x(1-x)\cos(3\pi\omega)$, which have negative Schwarzian derivative as $f_\omega'''=0$.

Note first that $\varphi^-=0$ and $\varphi^+=1$. Observing the symmetry properties of the invariant density of the driving map and of the cosine function, it is easy to check that 
not only $\lambda_\mac(\varphi^-)=\lambda_\mac(\varphi^+)<0$, but that indeed the log-Laplace transform $\psi$ is symmetric in the variables $\lambda^-$ and $\lambda^+$. It follows that
$t_*^-=t_*^+$ so that the convex function $\tilde{\psi}$ satisfies $\tilde{\psi}(-u)=\tilde{\psi}(u)$ and hence attains its minimum at $u_*=0$. Hence $\phi=h(u_*)$ is determined by
$\psi(t_*/2,t_*/2,h(u_*))=0$. As $\log|S'|$ is cohomologous to $\log2$, one can show that $\frac{\partial\psi}{\partial\lambda^S}(\lambda^-,\lambda^+,\lambda^S)=\log2$, so that $\phi=-\psi(t_*/2,t_*/2,0)$. In a diffusion approximation with $\lambda_\mac(\varphi^ \pm)$ close to zero this yields
$\phi\approx-\left(t_*\lambda_\mac(\varphi^\pm)+\frac{t_*^2}{4}(\sigma^2+c)\right)$, where ${\sigma^2\, c \choose \;c\ \; \sigma^2}$ denotes the asymptotic correlation matrix of the two components of the diffusion. Using the approximation $t_*\approx\frac{|\lambda_\mac(\varphi^\pm)|}{D}$ from Remark~\ref{remark:comparison1}, this yields
\begin{equation*}
\phi\approx\frac{\lambda_\mac(\varphi^\pm)^2}{D}-\frac{\lambda_\mac(\varphi^\pm)^2}{4D^2}(\sigma^2+c)=\frac{\lambda_\mac(\varphi^\pm)^2}{2D}-c\frac{\lambda_\mac(\varphi^\pm)^2}{4D^2}\ .
\end{equation*}
Neglecting the correlation term, this is precisely twice the value found in \cite{Ott1994a,Ott1993}, and this might explain why not only in \cite[Figure 13]{Ott1994a}, but also in \cite[Figure 6b (small $\kappa$)]{Pereira2008}, the numerically observed values for $\phi$ are roughly twice as large as the values suggested by the formula.
%
%
%
\end{remark}

\begin{lemma}\label{lemma:Laplace}
Let the processes $(Y_n)_{n\in\N}$ and  $(\tilde{Y}_n)_{n\in\N}$ be defined as in (\ref{eq:Yn-def}) and (\ref{eq:LDP-A-1}), respectively, and
denote by $\psi(\lambda)$ the asymptotic log-Laplace transform of $(Y_n)_{n\in\N}$ defined in (\ref{eq:psi-def}).
Then $\psi_m(\lambda):=\sum_{i=1}^m(t_i-t_{i-1})\psi(\lambda_i)$ is the asymptotic log-Laplace transform of $(\tilde{Y}_n)_{n\in\N}$. It is convex and smooth.
\end{lemma}
\begin{proof}
We argue by induction on $m$: For $m=1$ we only observe that $\psi_1=\psi$. For $m>1$ let $\lambda':=(\lambda_2,\dots,\lambda_m)\in(\R^3)^{m-1}$ and $t_i':=t_i-t_1$,  
\begin{equation}\label{eq:Y_n-inductive}
\tilde{Y}_n':=\left(Y_n(t_2'),Y_n(t_3')-Y_n(t_2'),\dots,Y_n(t_m')-Y_n(t_{m-1}')\right)\ .
\end{equation}
Then $\tilde{Y}_n=(Y_n(t_1),\tilde{Y}_n'\circ S^{[nt_1]})$ - at least up to small errors of oder $n^{-1}$ that are immaterial for the asymptotics. 
For a proof of this fact we recall that $nY_n$ is a partial sum process of the form $\sum_{k=0}^{n-1}V_k$. Hence, for $i=2,\dots,m$,
\begin{equation*}
\begin{split}
Y_n(t_i)-Y_n(t_{i-1})
&=
t_iY_{[nt_i]}-t_{i-1}Y_{[nt_{i-1}]}
=
\frac{t_i}{[nt_i]}\sum_{k=0}^{[nt_i]-1}V_k
-
\frac{t_{i-1}}{[nt_{i-1}]}\sum_{k=0}^{[nt_{i-1}]-1}V_k\\
&\approx
\frac{1}{n}\sum_{k=[nt_{i-1}]}^{[nt_i]-1}V_k
=
\frac{1}{n}\sum_{k=[nt_{i-1}]-[nt_1]}^{[nt_i]-[nt_1]-1}V_k\circ S^{[nt_1]}\\
&=
\frac{1}{n}\sum_{k=0}^{[nt_i]-[nt_1]-1}V_k\circ S^{[nt_1]}
-
\frac{1}{n}\sum_{k=0}^{[nt_{i-1}]-[nt_1]-1}V_k\circ S^{[nt_1]}\\
&\approx
\frac{t_i'}{[nt_i']}\sum_{k=0}^{[nt_i']-1}V_k\circ S^{[nt_1]}
-
\frac{t_{i-1}'}{[nt_{i-1}']}\sum_{k=0}^{[nt_{i-1}']-1}V_k\circ S^{[nt_1]}\\
\\
&=
\left(t_i'Y_{[nt_i']}-t_{i-1}'Y_{[nt_{i-1}']}\right)\circ S^{[nt_1]}
=
\left(Y_{n}(t_i')-Y_{n}(t_{i-1}'\right)\circ S^{[nt_1]}\ .
\end{split}
\end{equation*}

Now 
\begin{equation*}
\langle \lambda, \tilde{Y}_n\rangle
=
\lambda_1\,Y_n(t_1)+\langle\lambda',\tilde{Y}_n'\circ S^{[nt_1]}\rangle\ ,
\end{equation*}
so that
\begin{equation}\label{eq:last-one}
\begin{split}
\frac{1}{n}\log\int\exp\left(n\langle \lambda, \tilde{Y}_n\rangle\right)dm
&=
\frac{1}{n}\log\int\exp\left(nt_1\langle\lambda_1,Y_{[nt_1]}\rangle\right)\cdot
\exp\left(n\langle\lambda',\tilde{Y}_n'\rangle\right)\circ S^{[nt_1]}\,dm\\
&\approx
\frac{1}{n}\log\int \cL_{\lambda_1}^{[nt_1]}(1)\cdot
\exp\left(n\langle\lambda',\tilde{Y}_n'\rangle\right)dm
\end{split}
\end{equation}
where $\cL_{\lambda_1}$ denotes the Laplace-Perron-Frobenius operator with extra weight
\begin{equation*}
\exp\left(\lambda_{11}\log f_.'(\varphi^-)+\lambda_{12}\log f_.'(\varphi^+)+\lambda_{13}\log|S'|\right).
\end{equation*}
As $\left(e^{-\psi(\lambda_1)}\cL_{\lambda_1}\right)^k(1)$ converges uniformly to a strictly positive function as $k\to\infty$ (see e.g. \cite{Baladi2000}), the limit in (\ref{eq:last-one}) equals
\begin{equation*}
t_1\psi(\lambda_1)+\sum_{i=2}^m(t_i'-t_{i-1}')\psi(\lambda_i)
=
(t_1-t_0)\psi(\lambda_1)+\sum_{i=2}^m(t_i-t_{i-1})\psi(\lambda_i)
=
\psi_m(\lambda)\ ,
\end{equation*}
where we used the inductive hypothesis for the term
$\exp\left(n\langle\lambda',\tilde{Y}_n'\rangle\right)$.
\end{proof}


\end{document}